\def\todaysdate{4\textsuperscript{th} November 2020}
\documentclass[reqno,a4paper]{article}
\usepackage{etex}
\usepackage[T1]{fontenc}
\usepackage[utf8]{inputenc}
\usepackage{lmodern}

\usepackage[style=alphabetic,firstinits=true,backref,backend=bibtex,isbn=false,url=false,maxbibnames=99]{biblatex}

\renewbibmacro{in:}{}
\newbibmacro{string+doi}[1]{\iffieldundef{doi}{\iffieldundef{url}{#1}{\href{\thefield{url}}{#1}}}{\href{http://dx.doi.org/\thefield{doi}}{#1}}}
\DeclareFieldFormat*{title}{\usebibmacro{string+doi}{\emph{#1}}}
\DefineBibliographyStrings{english}{%
  backrefpage = {$\uparrow$},
  backrefpages = {$\uparrow$},
}

\usepackage{amssymb,amsmath,amsthm}
\usepackage{thmtools,xcolor}
\usepackage{units}
\usepackage{graphicx}
\usepackage[all]{xy}
\usepackage{tikz}
\usetikzlibrary{arrows,calc,positioning,decorations.pathreplacing}
\usepackage{relsize}
\usepackage{mhsetup}
\usepackage{mathtools}
\usepackage{float}
\usepackage{stmaryrd}
\usepackage{tocloft}
\usepackage{paralist} 
\usepackage[left=3cm,right=3cm,top=3cm,bottom=3cm]{geometry} 
\usepackage[bottom]{footmisc}
\usepackage[colorlinks,citecolor=blue,linkcolor=blue,urlcolor=blue,filecolor=blue,breaklinks]{hyperref}
\usepackage{footnotebackref}
\usepackage[format=plain,indention=1em,labelsep=quad,labelfont={up},textfont={footnotesize},margin=2em]{caption}
\usepackage[mathscr]{euscript}
\usepackage{accents}

\definecolor{lightblue}{rgb}{0.8,0.8,1}

\numberwithin{equation}{section}
\numberwithin{figure}{section}

\definecolor{vdarkred}{rgb}{0.7,0,0}
\declaretheoremstyle[
  spaceabove=\topsep,
  spacebelow=\topsep,
  headpunct=,
  numbered=no,
  postheadspace=1ex,
  headfont=\color{vdarkred}\normalfont\bfseries,
  bodyfont=\normalfont\itshape,
]{colored}
\declaretheoremstyle[
  spaceabove=\topsep,
  spacebelow=\topsep,
  headpunct=,
  numbered=no,
  postheadspace=1ex,
  headfont=\normalfont\bfseries,
  bodyfont=\normalfont\itshape,
]{italic}
\declaretheoremstyle[
  spaceabove=\topsep,
  spacebelow=\topsep,
  headpunct=,
  numbered=no,
  postheadspace=1ex,
  headfont=\normalfont\bfseries,
  bodyfont=\normalfont\upshape,
]{upright}

\declaretheorem[style=italic,name=Theorem,numbered=yes,numberwithin=section]{thm}
\declaretheorem[style=italic,name=Lemma,numbered=yes,numberlike=thm]{lem}
\declaretheorem[style=italic,name=Proposition,numbered=yes,numberlike=thm]{prop}
\declaretheorem[style=italic,name=Corollary,numbered=yes,numberlike=thm]{coro}

\declaretheorem[style=italic,name=Theorem,numbered=yes,numberwithin=section]{athm}

\declaretheorem[style=italic,name=Proposition,numbered=yes,numberlike=athm]{aprop}
\declaretheorem[style=italic,name=Corollary,numbered=yes,numberlike=athm]{acoro}

\declaretheorem[style=upright,name=Definition,numbered=yes,numberlike=thm]{defn}
\declaretheorem[style=upright,name=Remark,numbered=yes,numberlike=thm]{rmk}
\declaretheorem[style=upright,name=Example,numbered=yes,numberlike=thm]{eg}

\declaretheorem[style=upright,name=Notation,numbered=yes,numberlike=thm]{notation}

\declaretheorem[style=upright,name=Observation,numbered=yes,numberlike=thm]{observation}

\declaretheorem[style=upright,name=Assumption,numbered=yes,numberlike=thm]{assumption}

\makeatletter
\renewcommand*{\@seccntformat}[1]{\upshape\csname the#1\endcsname.\hspace{1ex}}
\renewcommand*{\section}{\@startsection{section}{1}{\z@}%
	{2.5ex \@plus 1ex \@minus 0.2ex}%
	{1.5ex \@plus 0.2ex}%
	{\normalfont\normalsize\bfseries}}
\renewcommand*{\subsection}{\@startsection{subsection}{2}{\z@}%
	{2.5ex \@plus 1ex \@minus 0.2ex}%
	{-1.5ex \@plus -0.2ex}%
	{\normalfont\normalsize\bfseries}}
\renewcommand*{\subsubsection}{\@startsection{subsubsection}{3}{\z@}%
	{2.5ex \@plus 1ex \@minus 0.2ex}%
	{-1.5ex \@plus -0.2ex}%
	{\normalfont\normalsize\bfseries}}
\renewcommand*{\paragraph}{\@startsection{paragraph}{4}{\z@}%
	{2.5ex \@plus 1ex \@minus 0.2ex}%
	{-1.5ex \@plus -0.2ex}%
	{\normalfont\normalsize\bfseries}}
\renewcommand*{\subparagraph}{\@startsection{subparagraph}{5}{\z@}%
	{2.5ex \@plus 1ex \@minus 0.2ex}%
	{-1.5ex \@plus -0.2ex}%
	{\normalfont\normalsize\slshape}}
\makeatother

\newcommand{\N}{\mathbb N}

\newcommand{\incl}[3][right]%
{%
\draw[<-,>=#1 hook] #2 to ($ #2!0.5!#3 $);
\draw[->,>=stealth'] ($ #2!0.5!#3 $) to #3;%
}
\newcommand{\inclusion}[5][right]%
{%
\draw[<-,>=#1 hook] #4 to ($ #4!0.5!#5 $) node[#2,font=\small]{#3};
\draw[->,>=stealth'] ($ #4!0.5!#5 $) to #5;%
}

\newenvironment{itemizeb}%
{\begin{compactitem}

}%
{\end{compactitem}}
{\begin{compactitem}[#1]

}%
{\end{compactitem}}
{\begin{compactdesc}

}%
{\end{compactdesc}}

\newcommand{\cB}{\mathcal{B}}

\newcommand{\cD}{\mathcal{D}}

\newcommand{\cG}{\mathcal{G}}

\newcommand{\cL}{\mathcal{L}}

\newcommand{\cN}{\mathcal{N}}

\newcommand{\cU}{\mathcal{U}}
\newcommand{\cV}{\mathcal{V}}

\newcommand{\bC}{\mathbb{C}}
\newcommand{\bD}{\mathbb{D}}

\newcommand{\bG}{\mathbb{G}}

\newcommand{\bI}{\mathbb{I}}

\newcommand{\bR}{\mathbb{R}}

\newcommand{\bU}{\mathbb{U}}
\newcommand{\bV}{\mathbb{V}}

\newcommand{\bZ}{\mathbb{Z}}

\renewcommand{\geq}{\geqslant}
\renewcommand{\leq}{\leqslant}
\renewcommand{\footnoterule}{%
  \kern -3pt
  \hrule width \textwidth height 0.4pt
  \kern 2.6pt
}

\newcommand{\bouter}{\ensuremath{\partial_{\mathrm{out}}}}
\newcommand{\binner}{\ensuremath{\partial_{\mathrm{in}}}}
\newcommand{\lf}{\ensuremath{\mathit{lf}}}
\newcommand{\lfb}{\ensuremath{\mathit{lfb}}}
\newcommand{\too}{\ensuremath{\longrightarrow}}
\newcommand{\longhookrightarrow}{\ensuremath{\lhook\joinrel\longrightarrow}}
\newcommand{\minner}{\ensuremath{M_{\mathrm{in}}}}
\newcommand{\mouter}{\ensuremath{M_{\mathrm{out}}}}

\DeclareMathOperator{\B}{\textit{\textbf{B}}}
\newcommand{\pairing}{\ensuremath{\langle\;\mathbin{,}\;\rangle}}

\definecolor{dgreen}{RGB}{0,150,0}

\begin{document}
\title{\Large\bfseries Lawrence-Bigelow representations, bases and duality\vspace{-1ex}}
\author{\small Cristina Anghel and Martin Palmer\quad $/\!\!/$\quad \todaysdate\vspace{-1ex}}
\date{}
\maketitle
{
\makeatletter
\renewcommand*{\BHFN@OldMakefntext}{}
\makeatother
}
\begin{abstract}
We study homological representations of mapping class groups, including the braid groups. These arise from the twisted homology of certain configuration spaces, and come in many different flavours. Our goal is to give a unified general account of the fundamental relationships (non-degenerate pairings, embeddings, isomorphisms) between the many different flavours of homological representations. Our motivating examples are the \emph{Lawrence-Bigelow representations} of the braid groups, which are of central importance in the study of the braid groups themselves, as well as their connections with quantum invariants of knots and links.
\end{abstract}
\tableofcontents


\section{Introduction}\label{introduction}

Homological representations of the braid groups, and of mapping class groups of surfaces, play a central role in the study of these groups. They also have many applications, especially as a bridge between the topological world and the world of quantum invariants of knots and links. In particular, the \emph{Lawrence-Bigelow representations} are one of the most important families of representations of the braid groups, in connection with quantum invariants.

The world of quantum invariants started with the landmark discovery of the Jones polynomial~\cite{Jones1985}. Subsequently, Reshetikhin and Turaev~\cite{ReshetikhinTuraev1991} defined a construction that starts with a quantum group and provides link invariants. Two important families of these \emph{quantum invariants} are the coloured Jones polynomials and the coloured Alexander polynomials, which recover the original Jones and Alexander polynomials as special cases. At around the same time, in 1990, Lawrence \cite{Lawrence1990HomologicalrepresentationsHecke} constructed a family of homological representations of braid groups, on the homology of certain coverings of configuration spaces on the punctured disc. This discovery has led to many deep and interesting connections between quantum topology and the homology of configuration spaces.

More specifically, Lawrence~\cite{Lawrence1993functorialapproachone-variable} and later Bigelow~\cite{Bigelow2002homologicaldefinitionJones}, based on her work, described a topological model for the original Jones polynomial, in terms of intersections of homology classes in coverings of configuration spaces. Kohno~\cite{Kohno2012Homologicalrepresentationsbraid, Kohno2017} proved that the quantum representations of braid groups are isomorphic to Lawrence's homological representations. Ito~\cite{Ito2015, Ito2016}, and later Martel~\cite{Martel2020}, constructed homological models for the coloured Alexander polynomials and coloured Jones polynomials respectively, as sums of traces of homological representations.

Returning to the question concerning topological models, in \cite{Anghel2017, Anghel2019} the first author proved that the coloured Jones and Alexander polynomials can be seen in terms of intersections of homology classes in versions of Lawrence representations. These are existence-type results, whose proofs are based on the existence of non-degenerate intersection pairings between certain homologies of covering spaces. The results of the present paper, which are concluded in Corollary \ref{CorADO} and Remark \ref{RkADO}, play an important role in the proof of ~\cite{Anghel2019}. Recently, the first author~\cite{Anghel2020, Anghel2020a} has given explicit unified topological models for the coloured Jones and coloured Alexander polynomials, constructing explicit homology classes given by Lagrangian submanifolds in versions of Lawrence representations. The result presented in Corollary \ref{CorADO} is used in these two constructions. In particular, the topological model of \cite{Anghel2020} for the family of coloured Jones polynomials recovers the topological model of Lawrence and Bigelow for the original Jones polynomial.

The Lawrence-Bigelow representations are also important due to their role in the proof of the fact, due to Bigelow~\cite{Bigelow2001Braidgroupslinear} and Krammer~\cite{Krammer2002Braidgroupslinear}, that the braid groups $\B_n$ are \emph{linear}, i.e., they each act faithfully on a finite-dimensional vector space. Specifically, the family of Lawrence-Bigelow representations is indexed by an integer $m\geq 1$, where $m=1$ corresponds to the classical (reduced) Burau representations~\cite{Burau1935}. The next member of this family, corresponding to $m=2$, is known as the \emph{Lawrence-Krammer-Bigelow representation} of $\B_n$, and was proven by Bigelow and Krammer to be faithful, from which linearity of the braid groups follows immediately. Bigelow's proof, in particular, uses the topological nature of this representation in an essential way, making geometric arguments possible.

The Lawrence-Bigelow representations have been studied further by Paoluzzi-Paris~\cite{PaoluzziParis2002noteLawrenceKrammer} and by Bigelow~\cite{Bigelow2004HomologicalrepresentationsIwahori}, who used them to construct representations of Hecke algbras. There are also many other constructions of homological representations of braid groups (for example an iterative construction due to Long and Moody~\cite{Long1994Constructingrepresentationsbraid}) and of surface braid groups~\cite{AnKo2010, BellingeriGodelleGuaschi2017}. In \cite{PalmerSoulie2019}, Souli{\'e} and the second author have unified these into a general machine for constructing homological representations of mapping class groups or motion groups (e.g.~braid groups) of any manifold.

A key aspect of \emph{homological} representations is that they are defined \emph{topologically}, through the action of the braid group or mapping class group on the (twisted) homology of configuration spaces. This makes it possible to approach them using topological tools, which provide insights that would be impossible to obtain purely algebraically. (For example, as mentioned above, this was the case in Bigelow's proof of the linearity of the braid groups.) Two of the most useful topological tools that one would like to have are:
\begin{itemizeb}
\item \emph{non-degenerate pairings} between different homological representations;
\item (geometrically-defined) \emph{embeddings} between different homological representations.
\end{itemizeb}
This is especially important since homological representations typically come in many different -- and subtly related -- flavours, which one often needs to pass between. Our goal in this paper is to give a unified general account of these basic topological tools for homological representations of mapping class groups of orientable surfaces.

Given an orientable surface $\Sigma$ and a decomposition of $\partial\Sigma$ into two pieces, an integer $m\geq 1$ and a local system on the configuration space $C_m(\Sigma)$, there are several basic flavours of homological representations, depending on which part of $\partial\Sigma$ one uses and which kind of homology one considers, for example:
\begin{itemizeb}
\item locally-finite (Borel-Moore) homology,
\item homology relative to the boundary,
\item ordinary homology (not relative),
\item locally-finite homology of the covering space associated to the local system.
\end{itemizeb}

Let us denote these, qualitatively, by $H^{\lf}$, $H^\partial$, $H$ and $H^{\lf,\sim}$ respectively. Depending on which half of the boundary of $\Sigma$ we use, we then have representations
\[
H^{\bullet}(\minner) \qquad\text{and}\qquad H^{\bullet}(\mouter)
\]
of the mapping class group of $\Sigma$, where $\bullet$ denotes any of the four decorations above. The case of the braid groups corresponds to taking $\Sigma = \Sigma_{0,n+1}$, the surface of genus $0$ with $n+1$ boundary circles, partitioned into one ``outer'' boundary circle and $n$ ``inner'' boundary circles.

Our results about these eight flavours of homological representations are as follows. (Precise details are given in \S\ref{results}.)

\begin{thm}[Theorem \ref{thm-pairings} and Remark \ref{rmk-bases}]
There are non-degenerate pairings
\begin{align*}
H^{\lf}(\minner) \otimes H^\partial(\minner) &\too R \\
H^{\lf}(\mouter) \otimes H^\partial(\mouter) &\too R.
\end{align*}
Moreover, these representations are free as $R$-modules. We describe explicit bases such that the pairings above are given by the identity matrix.
\end{thm}

For our next result, we assume the mild condition that the local system on the configuration space is \emph{$u$-homogeneous} for a unit $u \in R^\times$ (see Definition \ref{def-homogeneous}) and that the quantum factorials with respect to $u$ are non-zero-divisors in $R$.

\begin{thm}[Theorem \ref{thm-diagonal} and Corollary \ref{coro-embeddings}]
Under this mild assumption, there are embeddings of representations
\begin{align*}
H^\partial(\minner) &\too H^\lf(\mouter) \\
H^\partial(\mouter) &\too H^\lf(\minner).
\end{align*}
When $m>1$ this implies that $H^\lf(\mouter)$ and $H^\lf(\minner)$ are reducible. In fact, with respect to the explicit bases that we describe, the matrices of these embeddings are diagonal, and their diagonal entries are products of quantum factorials.
\end{thm}

Our next result is a mild generalisation of \cite[Theorem 3.1]{Kohno2017}, using the notion of \emph{genericity} of a local system (Definition \ref{def-generic}). This is a stronger assumption than that of being $u$-homogeneous, although one may always force a local system to become generic by taking a fibrewise tensor product; see Remark \ref{rmk:generic} for details.

\begin{thm}[Theorem \ref{thm-genericity}]
If the local system is generic, there are isomorphisms
\[
H(\minner) \;\cong\; H^\lf(\minner) \qquad\text{and}\qquad H(\mouter) \;\cong\; H^\lf(\mouter).
\]
\end{thm}

We next investigate the relation between $H^\lf$ and $H^{\lf,\sim}$. When working with ordinary homology, Shapiro's lemma implies that the homology of $X$ with coefficients in a local system arising from a covering $\widehat{X} \to X$ is isomorphic to the untwisted homology of the covering $\widehat{X}$ itself. However, for locally-finite (Borel-Moore) homology this is no longer true (see sections \ref{Shapiro} and \ref{Shapiro-lf} for a more detailed discussion).

\begin{thm}[Theorem \ref{thm-coverings}]
The natural maps
\begin{align*}
H^\lf(\minner) &\too H^{\lf,\sim}(\minner) \\
H^\lf(\mouter) &\too H^{\lf,\sim}(\mouter)
\end{align*}
are injective. Let ${}_\bullet$ denote either ${}_{\mathrm{in}}$ or ${}_{\mathrm{out}}$. If $\cB$ denotes a free basis for $H^\lf(M_\bullet)$ as a module over $R = k[G]$, then $H^{\lf,\sim}(M_\bullet)$ is a direct sum over $\cB$ of copies of the completion $k[[G]]$ of $k[G]$.
\end{thm}

See Definition \ref{def:completion} for the definition of $k[[G]]$.

Finally, in \S\ref{formula} we use the injections of Theorem \ref{thm-coverings} and Shapiro's lemma to re-interpret the non-degenerate pairings of Theorem \ref{thm-pairings} and the embeddings of Theorem \ref{thm-diagonal} in terms of the homology (locally-finite and relative to the boundary) of covering spaces of configuration spaces. Under this interpretation, we give an explicit geometric formula \eqref{eq:geometric-formula} for the non-degenerate pairings.

{\bf Acknowledgments.} C.~Anghel acknowledges the support of the European Research Council (ERC) under the European Union's Horizon 2020 research and innovation programme (grant agreement No 674978).

\section{Results}\label{results}

\begin{defn}
Let $(\Sigma;\binner,\bouter)$ be a \emph{surface triad}, by which we mean a compact, connected, orientable surface $\Sigma$ equipped with a decomposition $\partial \Sigma = \binner \cup \bouter$ into non-empty $1$-dimensional submanifolds $\binner , \bouter \subseteq \partial \Sigma$ such that $\partial \binner = \partial \bouter = \binner \cap \bouter$. We then consider the following unordered configuration spaces:
\begin{align*}
\minner &= C_m(\mathrm{int}(\Sigma) \cup \mathrm{int}(\binner)) = C_m(\Sigma \smallsetminus \bouter) \\
\mouter &= C_m(\mathrm{int}(\Sigma) \cup \mathrm{int}(\bouter)) = C_m(\Sigma \smallsetminus \binner).
\end{align*}
\end{defn}

\begin{rmk}
Note that $\minner$ and $\mouter$ are topological $2m$-manifolds with boundary. A configuration of $\minner$ lies in $\partial\minner$ exactly when at least one configuration point lies in $\binner$, and similarly a configuration of $\mouter$ lies in $\partial\mouter$ exactly when at least one configuration point lies in $\bouter$.
\end{rmk}

\begin{defn}
Let $\mathrm{Diff}(\Sigma,\bouter)$ be the topological group of diffeomorphisms of $\Sigma$ that restrict to the identity on a neighbourhood of $\bouter \subseteq \partial \Sigma$. Its discrete group of path-components is the mapping class group
\[
\Gamma(\Sigma,\bouter) = \pi_0(\mathrm{Diff}(\Sigma,\bouter)).
\]
Note that this group is isomorphic to the mapping class group of $\Sigma \smallsetminus \binner$ fixing its entire boundary $\partial(\Sigma \smallsetminus \binner) = \mathrm{int}(\bouter)$ pointwise.

Any diffeomorphism $\psi \in \mathrm{Diff}(\Sigma;\bouter)$ preserves the decomposition $\partial\Sigma = \binner \cup \bouter$, so there is a natural continuous action of $\mathrm{Diff}(\Sigma;\bouter)$ on $\minner$ and on $\mouter$.
\end{defn}

\begin{eg}\label{eg:1}
In our first motivating example we have $\Sigma = \Sigma_{0,n+1}$, the closed $2$-disc after removing $n\geq 2$ open subdiscs with pairwise disjoint closures, and
\begin{itemizeb}
\item $\bouter = \partial \bD^2$ is the outer boundary circle (the boundary of the $2$-disc);
\item $\binner$ is the union of the $n$ inner boundary circles (the boundaries of the $n$ open discs that we have removed).
\end{itemizeb}
In this case, the mapping class group $\Gamma(\Sigma_{0,n+1},\partial \bD^2)$ is naturally isomorphic to the braid group $\B_n$ on $n$ strands.
\end{eg}

\begin{eg}\label{eg:2}
Our second motivating example is a small, but non-trivial, variation of our first example. We again take $\Sigma = \Sigma_{0,n+1}$ for $n\geq 2$, but now
\begin{itemizeb}
\item $\bouter = \bI$ is a closed interval in the outer boundary circle;
\item $\binner$ is the complementary closed interval in the outer boundary circle, together with all $n$ inner boundary circles.
\end{itemizeb}
For diffeomorphisms of surfaces, the condition of fixing a boundary circle is equivalent -- up to isotopy -- to the condition of fixing a single point, or fixing an interval, in that boundary circle. As a result, the mapping class group $\Gamma(\Sigma_{0,n+1} , \bI)$ is again naturally isomorphic to the braid group $\B_n$ on $n$ strands.
\end{eg}

\begin{rmk}\label{rmk-local-systems}
Write $M = C_m(\Sigma)$. Then the two inclusions
\[
\minner \longrightarrow M \longleftarrow \mouter
\]
are both homotopy equivalences. Thus, local systems on $\minner$ are in one-to-one correspondence with local systems on $\mouter$ and also in one-to-one correspondence with local systems on $M$.
\end{rmk}

\begin{defn}
Recall that a \emph{rank-1 local system} on a space $X$ is a bundle of $R$-modules over $X$ such that the fibre over each point of $X$ is isomorphic to $R$ as a rank-$1$ free module over itself.\footnote{We assume that $X$ is locally path-connected and semi-locally simply-connected (which will always be the case in our examples, where $X$ is always a manifold).} Equivalently, when $X$ is based and path-connected, it may be thought of as an action of $\pi_1(X)$ on $R$ by $R$-module automorphisms, in other words a homomorphism $\pi_1(X) \to GL_1(R) = R^\times$.
\end{defn}

Let us now fix a unital ring $R$ equipped with an anti-automorphism $\alpha \colon R \to R$ with $\alpha^2 = \mathrm{id}$ and a rank-1 local system $\cL$ on $M$ defined over $R$. We make the following assumptions:

\begin{assumption}\label{assumption}
The local system $\cL$ is preserved by the action of $\mathrm{Diff}(\Sigma;\bouter)$ on $M$. In other words, for each $\psi \in \mathrm{Diff}(\Sigma;\bouter)$ we have $\psi^*(\cL) \cong \cL$. Moreover, we require that the homomorphism $\pi_1(X) \to R^\times$ is compatible with the anti-automorphism $\alpha$ restricted to $R^\times$ and the canonical anti-automorphism of the group $\pi_1(X)$ given by $g \mapsto g^{-1}$.
\end{assumption}

\begin{eg}\label{eg:local-system}
In Examples \ref{eg:1} and \ref{eg:2}, we have $M = C_m(\Sigma_{0,n+1})$, so there are homomorphisms
\[
\varphi \colon \pi_1(M) \too \B_{m,n} \too \begin{cases}
\bZ^2 = \langle c , x \rangle & \text{if } m = 1; \\
\bZ^3 = \langle c , x , d \rangle & \text{if } m \geq 2.
\end{cases}
\]
The first homomorphism is given by filling in each of the $n$ inner boundary-components of $\Sigma_{0,n+1}$ with a disc containing one new configuration point, to obtain a configuration of $m+n$ points in $\bD^2$, partitioned as $\{m,n\}$. The second homomorphism is the abelianisation of $\B_{m,n}$. The generator $c$ interchanges two of the last $n$ points; $d$ interchanges two of the first $m$ points and $x$ sends the $m$-th point once around the $(m+1)$st point. The image of the composite homomorphism $\varphi$ is the subgroup $\bZ = \langle x \rangle$ for $m=1$ and $\bZ^2 = \langle x , d \rangle$ for $m \geq 2$. This quotient $\varphi$ of $\pi_1(M)$ therefore determines a regular covering
\[
M^\varphi \too M
\]
with deck transformation group $G = \bZ = \langle x \rangle$ when $m=1$ and $G = \bZ^2 = \langle x , d \rangle$ when $m \geq 2$. Then taking free abelian groups fibrewise, we obtain a rank-1 local system
\begin{equation}
\label{eq:rank-1-system}
\bZ[M^\varphi] \too M
\end{equation}
defined over the group-ring $\bZ[G]$. One may easily verify that the quotient map $\varphi \colon \pi_1(M) \to G$ is \emph{invariant} under the induced action of $\mathrm{Diff}(\Sigma;\bouter)$ on $\pi_1(M)$. This implies that Assumption \ref{assumption} is satisfied for this local system.
\end{eg}

\begin{eg}\label{eg:tensor-product}
In Example \ref{eg:local-system} we described a rank-1 local system \eqref{eq:rank-1-system} over $\bZ[G]$. For any other unital ring $R$ and $(\bZ[G],R)$-bimodule $V$, we may take the fibrewise tensor product
\begin{equation}
\label{eq:rank-1-system-tensor}
\bZ[M^\varphi] \otimes_{\bZ[G]} V \too M,
\end{equation}
which is then a rank-1 local system over $R$. For example, suppose we have a homomorphism
\[
\theta \colon G \too \bC^\times
\]
(since $G$ is either $\bZ$ or $\bZ^2$ this amounts to a choice of either one or two non-zero complex numbers $\theta(x)$ and $\theta(d)$). This extends by linearity to a ring homomorphism, also denoted $\theta \colon \bZ[G] \to \bC$. We may then take $R = \bC$ and $V = \bC$ considered as a right module over itself (by right-multiplication) and as a left module over $\bZ[G]$ via $\theta$. In this case, \eqref{eq:rank-1-system-tensor} is a rank-1 local system
\begin{equation}
\label{eq:rank-1-system-C}
\bZ[M^\varphi] \otimes_{\theta} \bC \too M
\end{equation}
over $\bC$. These local systems inherit from \eqref{eq:rank-1-system} the property that they satisfy Assumption \ref{assumption}.
\end{eg}

\begin{notation}
\emph{Homology groups of $\minner$ and of $\mouter$ -- ordinary or locally finite (Borel-Moore) -- will always be taken with coefficients in the local system $\cL$ unless otherwise specified.}
\end{notation}

\begin{observation}
The first part of Assumption \ref{assumption} implies that the four $R$-modules
\begin{equation}
\label{eq:representations}
H_m^{\lf}(\minner), \qquad H_m(\minner;\partial \minner), \qquad H_m^{\lf}(\mouter), \qquad H_m(\mouter,\partial \mouter)
\end{equation}
are representations of $\Gamma(\Sigma;\bouter)$. In the setting of Example \ref{eg:1}, together with the local system described in Example \ref{eg:local-system}, these are (different versions of) the $m$-th \emph{Lawrence-Bigelow representations} of the braid groups.
\end{observation}

\begin{athm}\label{thm-pairings}
There is a $\Gamma(\Sigma;\bouter)$-invariant pairing
\[
H_m^{\lf}(\minner) \otimes H_m(\minner;\partial \minner) \too R
\]
that restricts to a \textbf{non-degenerate} $\Gamma(\Sigma;\bouter)$-invariant pairing
\begin{equation}
\label{eq:non-degenerate-pairing}
H_m^{\lf}(\minner) \otimes H_m^\partial(\minner) \too R,
\end{equation}
where $H_m^\partial(\minner)$ is a sub-representation of $H_m(\minner;\partial \minner)$ that we describe explicitly. The same statements hold when $\minner$ is replaced with $\mouter$.
\end{athm}

\begin{rmk}\label{rmk-bases}
We also describe explicit free bases for the four $R$-modules
\begin{equation}\label{four-modules}
H_m^{\lf}(\minner), \qquad H_m^\partial(\minner), \qquad H_m^{\lf}(\mouter), \qquad H_m^\partial(\mouter);
\end{equation}
see \S\ref{pairings}, with respect to which the pairings \eqref{eq:non-degenerate-pairing} of Theorem \ref{thm-pairings} are given by identity matrices.
\end{rmk}

\begin{defn}\label{def-homogeneous}
Say that the local system $\cL$ is \emph{homogeneous} if, whenever $\gamma$ is a loop of configurations in which all points remain fixed except two, which swap places while staying within a small subdisc of $\Sigma$, the monodromy around $\gamma$ is a fixed element of $R^\times$. In other words, viewing the local system as a homomorphism $\pi_1(M) \to R^\times$, all loops $\gamma$ of the form described above are mapped to a fixed unit in $R^\times$. If we denote this unit by $u \in R^\times$, we say that $\cL$ is \emph{$u$-homogeneous}.

Of course, this definition is vacuous if $m=1$, since there are no loops $\gamma$ as described; in this case we take the convention that every local system is $1$-homogeneous.
\end{defn}

\begin{eg}\label{eg:homogeneous}
The trivial local system over $R$ is $1$-homogeneous. In Example \ref{eg:local-system} we have
\[
R = \bZ[G] = \begin{cases}
\bZ[x^{\pm 1}] & \text{if } m=1; \\
\bZ[x^{\pm 1},d^{\pm 1}] & \text{if } m \geq 2,
\end{cases}
\]
and the local system \eqref{eq:rank-1-system} is $d$-homogeneous (or $1$-homogeneous when $m=1$). In Example \ref{eg:tensor-product}, the local system \eqref{eq:rank-1-system-C} defined over $\bC$ is $\theta(d)$-homogeneous.
\end{eg}

\begin{defn}
For a unital ring $R$, a unit $u \in R^\times$ and an integer $n \geq 1$, the \emph{quantum integers} and \emph{quantum factorials} are defined as follows:
\begin{align*}
[n]_u &= 1+u+u^2 + \cdots + u^{n-1} \in R \\
[n]_u! &= [1]_u \cdot [2]_u \cdots [n-1]_u \cdot [n]_u \in R.
\end{align*}
For example, if $u=1$ this is the classical factorial $n!$.
\end{defn}

\begin{athm}\label{thm-diagonal}
There are $\Gamma(\Sigma;\bouter)$-equivariant $R$-module homomorphisms
\begin{equation}\label{injections}
\begin{aligned}
H_m^\partial(\minner) &\too H_m^{\lf}(\mouter) \\
H_m^\partial(\mouter) &\too H_m^{\lf}(\minner),
\end{aligned}
\end{equation}
whose matrices, with respect to our explicit bases, are diagonal. If the local system is $u$-homogeneous, the diagonal entries of these matrices are all products of quantum factorials $[r]_u!$.
\end{athm}

This immediately implies:

\begin{acoro}\label{coro-embeddings}
If the local system $\cL$ is $u$-homogeneous and the quantum factorials $[r]_u!$ are all non-zero-divisors in $R$, the homomorphisms \eqref{injections} are embeddings of $\Gamma(\Sigma;\bouter)$-representations over $R$. In particular, the $\Gamma(\Sigma;\bouter)$-representations $H_m^{\lf}(M)$ and $H_m^{\lf}(M')$ are \textbf{reducible} for $m>1$.
\end{acoro}

\begin{rmk}
The condition on quantum factorials is satisfied whenever $R$ is an integral domain and $u \in R^\times$ has infinite order. This holds in our key examples of local systems $\cL$ (see Example \ref{eg:homogeneous}), as long as, in the case of Example \ref{eg:tensor-product}, the complex number $\theta(d) \in \bC^*$ is not a root of unity.
\end{rmk}

\begin{defn}\label{def-generic}
The local system $\cL$ is called \emph{generic for $\minner$} if it satisfies the following condition. Let $\gamma$ be an unbased loop in $\minner$ that may be homotoped to be disjoint from any given compact subset. Then the monodromy $m_\gamma \in R^\times$ of $\cL$ around $\gamma$ (for any choice of basepoint on $\gamma$) has the property that $1 - m_\gamma$ is also a unit of $R$. The property of being \emph{generic for $\mouter$} is similar.
\end{defn}

\begin{aprop}\label{thm-genericity}
If $\cL$ is generic for $\minner$ then the natural morphism of $\Gamma(\Sigma;\bouter)$-representations
\[
H_m(\minner) \too H_m^{\lf}(\minner)
\]
is an isomorphism. Similarly with $\minner$ replaced by $\mouter$ if $\cL$ is generic for $\mouter$.
\end{aprop}

\begin{rmk}
This fact is essentially due to Kohno \cite[Theorem 3.1]{Kohno2017}; we just explain how his proof generalises to allow local systems defined over any unital ring $R$ rather than just the complex numbers $\bC$.
\end{rmk}

\begin{rmk}\label{rmk:generic}
Our key examples of local systems are sometimes generic and sometimes not. In particular, the local system \eqref{eq:rank-1-system} is \emph{not} generic, since the elements $1-x$ and $1-d$ are not invertible in the Laurent polynomial ring in the variables $x$ and $d$.

However, taking the fibrewise tensor product as in Example \ref{eg:tensor-product} can force a local system to be generic. For example, we could take $R$ to be the localisation of the Laurent polynomial ring with respect to the multiplicative subset generated by $\{ 1-x , 1-d \}$, then take $V=R$ as a right-module over itself and as a left-module over the Laurent polynomial ring. The resulting local system \eqref{eq:rank-1-system-tensor} is generic.

Alternatively, the local system \eqref{eq:rank-1-system-C} defined over $\bC$ is generic whenever the homomorphism $\theta$ sends $x$ and $d$ to elements of $\bC \smallsetminus \{0,1\}$.
\end{rmk}

\begin{assumption}
We now assume that the rank-1 local system $\cL$ arises from a regular covering. More precisely, let $\varphi \colon \pi_1(M) \to G$ be a surjective homomorphism, defining a regular covering $M^\varphi \to M$ with deck transformation group $G$. Choose a ring $k$ and take free $k$-modules fibrewise to define a bundle $k[M^\varphi] \to M$ of $k[G]$-modules. This is a rank-1 local system defined over the group-ring $R = k[G]$. \emph{We assume that $\cL$ is of this form}.
\end{assumption}

There is a natural $R$-module homomorphism
\begin{equation}\label{base-to-covering}
H_m^{\lf}(\minner;\cL) \too H_m^{\lf}((\minner)^\varphi;k),
\end{equation}
where $(\minner)^\varphi$ denotes the restriction of the covering $M^\varphi \to M$ to $\minner \subseteq M$, and the homology group on the right-hand side is taken with trivial $k$-coefficients.

\begin{rmk}
If we were taking \emph{ordinary} homology, rather than locally finite (Borel-Moore) homology, this homomorphism would be an isomorphism, by Shapiro's lemma for covering spaces. However, the homomorphism \eqref{base-to-covering} is in general not an isomorphism if the covering space $M^\varphi \to M$ is infinite-sheeted. This is discussed in more detail in sections \ref{Shapiro} and \ref{Shapiro-lf}.
\end{rmk}

\begin{defn}\label{def:completion}
Define $k[[G]]$ to be the $k[G]$-module $k^G$ of functions $G \to k$, in other words the product of $\lvert G \rvert$ copies of $k$. This contains the group-ring $k[G]$:
\begin{equation}\label{completed-group-ring}
k[G] = \bigoplus_{g \in G} k \longhookrightarrow \prod_{g \in G} k = k[[G]].
\end{equation}
\end{defn}

\begin{athm}\label{thm-coverings}
Let $\cB_{\mathrm{in}}$ be the free basis over $R = k[G]$ for $H_m^{\lf}(\minner;\cL)$ mentioned in Remark \ref{rmk-bases}. Then the $R$-module $H_m^{\lf}((\minner)^\varphi;k)$ is isomorphic to a direct sum of copies of $k[[G]]$ indexed by the set $\cB_{\mathrm{in}}$. Under this identification, the homomorphism \eqref{base-to-covering} is given by $\cB_{\mathrm{in}}$ copies of the natural inclusion \eqref{completed-group-ring}. In other words:
\[
H_m^{\lf}(M;\cL) \cong \bigoplus_\cB k[G] \longhookrightarrow \bigoplus_\cB k[[G]] \cong H_m^{\lf}(M^\varphi;k).
\]
An identical statement also holds when $\minner$ is replaced with $\mouter$, and $\cB_{\mathrm{in}}$ is replaced by $\cB_{\mathrm{out}}$.
\end{athm}

In \S\ref{elements} we describe geometrically some interesting elements of $H_m^{\lf}((\minner)^\varphi;k)$ that do not lie in the image of \eqref{base-to-covering}.

\paragraph{Outline of the paper.}

We prove Theorem \ref{thm-pairings} in \S\ref{pairings}, where we also describe the explicit free bases mentioned in Remark \ref{rmk-bases}. Theorem \ref{thm-diagonal} and Corollary \ref{coro-embeddings} are then proven in \S\ref{embeddings}. In \S\ref{genericity} we discuss \emph{genericity} of local systems and prove Theorem \ref{thm-genericity}. Theorem \ref{thm-coverings}, comparing locally-finite homology of infinite coverings with the corresponding twisted locally-finite homology of the base space, is proven in \S\ref{coverings}, in two parts: Theorem \ref{thm-lf-basis} describes the locally-finite homology of the base space and Theorem \ref{thm-lf-basis-covering} describes the locally-finite homology of the covering space and the map between the two. Finally, in \S\ref{formula} we reinterpret the pairings of Theorem \ref{thm-pairings} in terms of the homology of covering spaces and describe a concrete formula for computing these pairings.

\section{Pairings and bases}\label{pairings}

Now, we will show that once we fix a splitting of the boundary of the configuration space, the two homologies relative to each part of this splitting are related by an intersection form.
\begin{thm}[The first part of Theorem \ref{thm-pairings}]\label{T:1}
There exists an intersection pairing 
\begin{equation}
\pairing \colon H_m^{\lf}(\minner) \otimes H_m(\minner, \partial\minner) \longrightarrow R
\end{equation}
which is linear with respect to the $R$-action on the first component twisted by the involution $\alpha$.
\end{thm}
\begin{proof}
In order to show this, we will use two main homological tools. The first uses the fact that $\minner$ is a connected, orientable $2m$-dimensional manifold.
\begin{prop}[Poincaré duality for twisted homology.]
For any $k\in \N$ such that $k \leq 2m$ we have an isomorphism:
\begin{equation}
p_k \colon H^{k}(\minner,\binner) \longrightarrow H^{\lf}_{2m-k}(\minner).
\end{equation}
\end{prop}

\begin{lem}[Relative cap product for twisted homology.]
Let $k,l \in \N$ such that $0 \leq k\leq l \leq 2m$. Then, we have the following homomorphism:
\begin{equation}
\cap_{k,l} \colon H^{k}(\minner,\binner) \otimes H_{l}(\minner,\binner) \longrightarrow H_{l-k}(\minner).
\end{equation}
\end{lem}

Now, we combine these two results to define the pairing $\pairing$, which is given by the following composition:
\begin{center}
\begin{tikzpicture}
[x=1.2mm,y=1.4mm]

\node (b1) at (-10,30)    {$H_m^{\lf}(\minner) \otimes H_m(\minner, \binner)$};
\node (b2) at (-10,15)   {$H^m(\minner,\binner) \otimes H_m(\minner, \binner)$};
\node (b3) at (-10,0)   {$H_0(\minner)\cong R$};

\draw[->] (b1) to node[left,xshift=-2mm,font=\normalsize]{$p_m^{-1}\otimes \mathrm{Id}$} (b2);
\draw[->] (b2) to node[left,xshift=-2mm,font=\normalsize]{$\cap_{m,m} $} (b3);
\draw[->, in=30, out=-45] (b1.east) to node[right,font=\normalsize]{$\pairing$}   (b3.east);
\end{tikzpicture}
\end{center}
It follows that the pairing $\pairing$ has the required properties.
\end{proof}

\begin{rmk}\label{rmk:choice}
Implicitly, the construction of the pairing above depends on a choice of fundamental class $[\minner] \in H_{2m}^{\lf}(\minner)$ in the step where we use Poincar{\'e} duality. Any two such choices differ by a unit in the underlying ring $R$. Later, when we compute this pairing on explicit homology classes, this choice will correspond, geometrically, to a choice of a collection of paths in the surface from one component of $\bouter$ to each component of $\binner$.
\end{rmk}

We now make a mild simplifying assumption:

\begin{assumption}\label{assumption-boundary}(Partition of the boundary of the surface)\\
We assume that $\bouter \subseteq \partial \Sigma$ is contained in a single component of $\partial \Sigma$. This means that we may draw the surface $\Sigma$ as in Figure \ref{fig-basis-in} (ignore the blue and green arcs for the moment), where $\binner$ is coloured red and $\bouter$ is coloured black.
\end{assumption}

\begin{defn}(Surface)\\
Suppose that $\Sigma$ has $n+1$ boundary circles, and that $\binner$ consists of $n$ of these boundary circles together with $r$ disjoint subintervals of the $(n+1)$st boundary circle. Write $g$ for the genus of $\Sigma$. 
\end{defn}
\begin{notation}(Set of partitions) Let us consider the following indexing set:
\[
E_{l,m}=\{e=(e_1,...,e_{l})\in \N^{l} \mid e_1+...+e_{l}=m \}.
\]
\end{notation}

We will use the set of partitions of $m$ into $l = n-1+k+2g$ pieces.

\subsection{Pairing between $H_m^{\lf}(\minner)$ and $ H_m(\minner, \partial \minner)$.}

\

In this part, we will consider two subspaces in the homologies
\[
H_m^{\lf}(\minner) \qquad\text{ and }\qquad H_m(\minner, \partial \minner)
\]
which are generated by certain submanifolds in the configuration space. Then, we aim to study the precise form of the intersection pairing between these subspaces.

\begin{defn}(Homology classes in $H_m(\minner,\partial \minner)$)\label{D:1}\\
For the first part, for each partition $e=(e_1,...,e_{2g},e'_1,...,e'_{n-1},e''_1,...,e_{k})\in E_{n-1+k+2g,m}$, we construct a homology class $ \cU_e \in H_m(\minner, \partial \minner)$. In order to do this, we draw $m$ red segments on our surface, which are prescribed by this partition, as in figure \ref{fig-basis-in}. More specifically, we have the following three families of curves:
\begin{itemize}
\item $\forall j \in \{1,..,k\} $ we draw $f_j$ parallel segments between the first boundary disk and the $j^{th}$ blue interval in the boundary;
\item  $\forall i \in \{1,..,n-1\} $ we draw $f'_i$ parallel segments between the $i^{th}$ boundary disk and the ${i+1}^{st}$ blue boundary disk;
\item $\forall s \in \{1,..,2g\} $ we draw $f''_s$ parallel arcs between the $n^{th}$ boundary disk and itself, passing through the core of the handle numbered by $s$.
\end{itemize} 
Now, we denote by $\bU_{e}$ the $m-$dimensional submanifold in $\minner$ given by the subspace consisting of configurations where exactly one point lies on each red arc. Since this is orientable and its boundary lies in $\partial\minner$, it has a fundamental class in relative homology, which we denote by 
\[
\cU_e \in H_m(\minner,\partial \minner).
\]
\end{defn}

\begin{defn}\label{def:bin}
(Homology classes $H^{\lf}_m(\minner)$)\label{D:2}\\
Secondly, let $f=(f_1,...,f_{2g},f'_1,...,f'_{n-1},f''_1,...,f''_{k})\in E_{n-1+k+2g,m}$ be a partition.  From this, we define a homology class $ \cD_f \in H_m^{\lf}(\minner)$. In order to do this, we will use configuration spaces on the green segments on our surface, which are prescribed by this partition, as in figure \ref{fig-basis-in}. More specifically:
\begin{itemize}
\item $\forall j \in \{1,..,k\} $ we consider the unordered configuration space of $f_j$ points on a green semicircle around the $j^{th}$ blue segment of the outer boundary;
\item  $\forall i \in \{1,..,n-1\} $ we consider the unordered configuration space of $f'_i$ points on a vertical green segment which lies between the $i^{th}$ boundary disk and the ${(i+1)}^{st}$ blue boundary disk;
\item $\forall s \in \{1,..,2g\} $ we consider the unordered configuration space of  $f''_s$ points on the green segment which is perpendicular to the core of the handle numbered by $s$.
\end{itemize} 
We then define $\bD_{f}$ to be the $m-$dimensional submanifold of $\minner$ consisting of configurations where the prescribed number of points lies on each of the green arcs. Since this is orientable and properly embedded, it has a fundamental class in locally-finite homology, which we denote by
\[
\cD_f \in H_m^{\lf}(\minner).
\]
Moreover, we write
\[
\cB_{\mathrm{in}} = \{ \cD_f \mid f \in E_{n-1+k+2g,m} \} .
\]
\end{defn}

\begin{figure}[ht]
\centering
\[
{\color{red} \cU_e\in H^{\partial}_m(\minner)} \hspace{10mm} \text{ and } \hspace{10mm} {\color{dgreen} \cD_f \in  \bar{H}^{\lf}_m(\minner)}
\]
\includegraphics[scale=0.6]{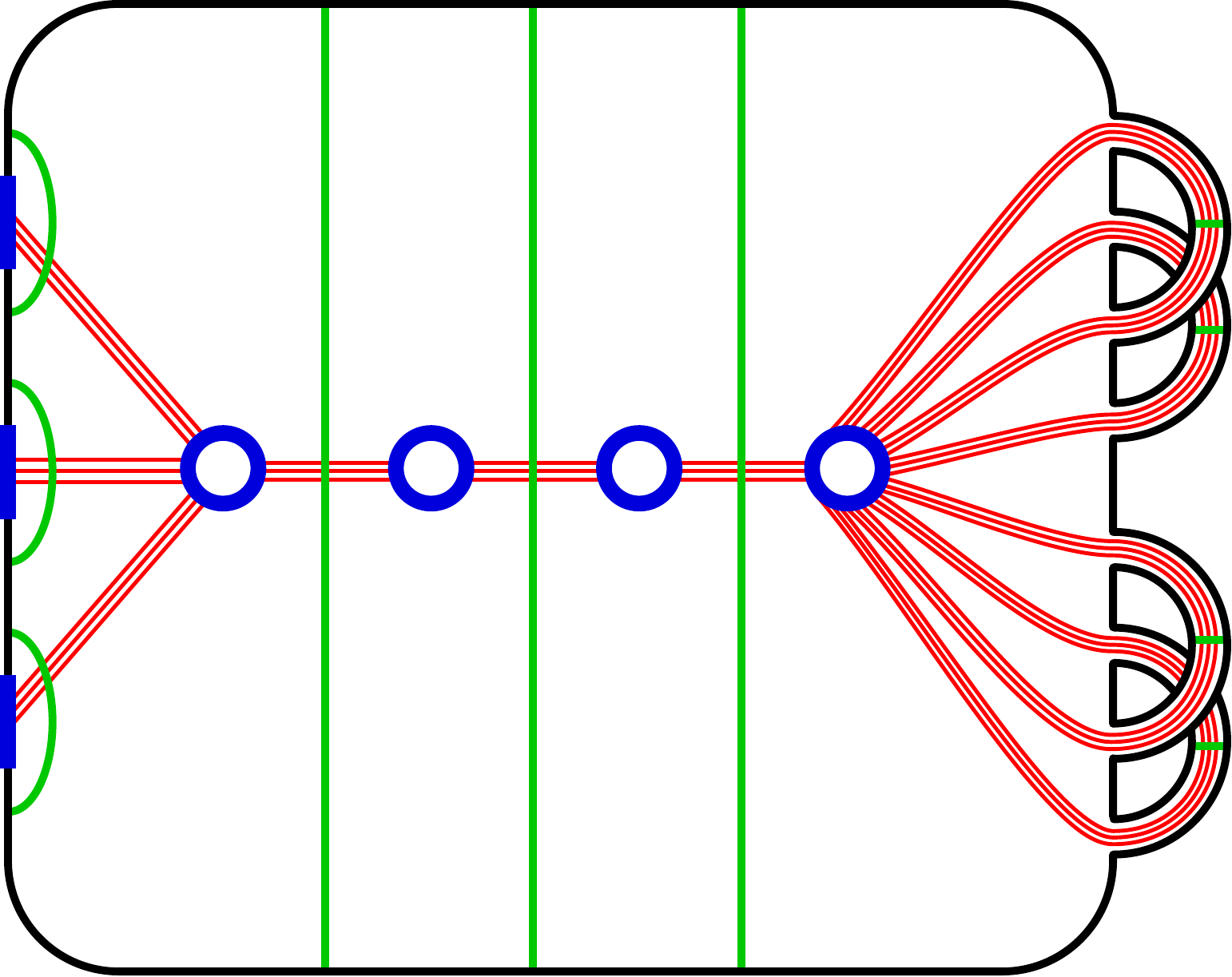}
\caption{The basis for $\bar{H}^{\lf}_m(\minner)$ is in \textcolor{dgreen}{green} and the basis for $H^{\partial}_m(\minner)$ is in \textcolor{red}{red}.}
\label{fig-basis-in}
\end{figure}

\begin{defn}(Subspaces generated by these homology classes)\\
Let us consider the subspaces given by these particular elements and denote them as below:
\begin{equation}\label{eq:bases}
\begin{aligned}
H_m^{\partial}(\minner) &\coloneqq \langle \cU_e \mid e \in E_{n-1+k+2g,m} \rangle \subseteq H_m(\minner,\partial \minner)\\
\bar{H}^{\lf}_m(\minner) &\coloneqq \langle \cD_f \mid f \in E_{n-1+k+2g,m} \rangle \subseteq H^{\lf}_m(\minner).
\end{aligned}
\end{equation}
\end{defn}

\begin{lem}[Diagonal intersection form]\label{L:1}\hspace{0pt}\\
The intersection pairing is diagonal on these subspaces, in the sense that we have:
\begin{equation}
\begin{aligned}
\pairing &\colon \bar{H}^{\lf}_m(\minner) \otimes H_m^{\partial}(\minner) \longrightarrow R\\
& \qquad\quad \langle \cD_f, \cU_e \rangle= \delta_{e,f},
\end{aligned}
\end{equation}
where $\delta$ denotes the Kronecker symbol. As a consequence, the sets of elements $\{\cU_e\}$ and $\{\cD_f\}$ are linearly independent, so they are bases for the free $R$-modules $H_m^{\partial}(\minner)$ and $\bar{H}^{\lf}_m(\minner)$.
\end{lem}

\begin{proof}
Let $e,f \in E_{n-1+k+2g,m}$. Suppose we have an intersection point $x \in \cD_f \cap \cU_e$. This means that $x$ is a multipoint in the configuration space, having $m$ components. In particular, since $x \in \cD_f$ this means that its components are distributed on the green segments as below:  
\begin{itemize}
\item $\forall j \in \{1,..,k\} $, $x$ has $f_j$ points on the green semicircle around the $j^{th}$ blue segment in the boundary 
\item $\forall i \in \{1,..,n-1\} $, it has $f'_i$ points on the vertical green segment which lies between the $i^{th}$ boundary disk and the ${i+1}^{st}$ blue boundary disk
\item $\forall s \in \{1,..,2g\} $, there are $f''_s$ components on the green segment which ``cuts'' the handle numbered by $s$.
\end{itemize} 
On the other hand, all these points have to lie on the red segments as well. This means that we have the following inequalities:
\begin{equation}
\begin{aligned}
&f_j \geq e_j, \ \ \forall j \in \{1,..,k\}\\
&f'_i \geq e'_i, \ \ \forall i \in \{1,..,n-1\}\\
&f''_s \geq e''_s, \ \ \forall s \in \{1,..,2g\}.
\end{aligned}
\end{equation}
However, since the total sum is:
\begin{equation}
\begin{aligned}
&e_1+ \cdots +e_{2g}+e'_1+ \cdots +e'_{n-1}+e''_1+ \cdots +e''_{k}\\
&\quad =f_1+ \cdots +f_{2g}+f'_1+ \cdots +f'_{n-1}+f''_1+ \cdots +f''_{k}\\
&\quad =n-1+r+2g,
\end{aligned}
\end{equation}
it follows that the partitions coincide and so $e=f$.

We conclude that if the intersection form $\langle \cD_f, \cU_e \rangle \neq 0$ then $e=f$. Now, if the partitions coincide, it is easy to compute the orientations on the picture and see that 
$\langle \cD_f, \cU_e \rangle=1$, which concludes the proof.
\end{proof}

\subsection{Pairing between $H_m^{\lf}(\mouter)$ and $ H_m(\mouter, \partial \mouter)$}
\hspace{0pt}

Now, we construct two subspaces in the homologies of the configuration space based on the outer part of the boundary 
\[
H_m^{\lf}(\mouter) \hspace{10mm} \text{ and } \hspace{10mm} H_m(\mouter, \partial \mouter),
\]
generated by geometric elements, and we will be interested in the pairing between these. The arguments are analogous to those for the configuration spaces on $\minner$ considered above, so we outline below just the main points of the construction.  
\begin{defn}\label{def:bout}
(Elements in the two homologies)\\
For two partitions $e,f\in E_{n-1+k+2g,m}$, we define $\cG_f \in  \bar{H}^{\partial}_m(\mouter)$ and $\cV_e\in \bar{H}^{\lf}_m(\mouter)$ to be the homology classes given by the quotient of the products of the green segments and configuration spaces on the red segments respectively, from figure \ref{fig-basis-out}, prescribed by these partitions (constructed similarly to those of definitions \ref{D:1} and \ref{D:2}). In particular, we write
\[
\cB_{\mathrm{out}} = \{ \cV_e \mid e \in E_{n-1+k+2g,m} \} .
\]
\end{defn}

\begin{figure}[ht]
\centering
\[
{\color{red} \cV_e\in \bar{H}^{\lf}_m(\mouter)} \hspace{10mm} \text{ and } \hspace{10mm} {\color{dgreen} \cG_f \in  \bar{H}^{\partial}_m(\mouter)}
\]
\includegraphics[scale=0.6]{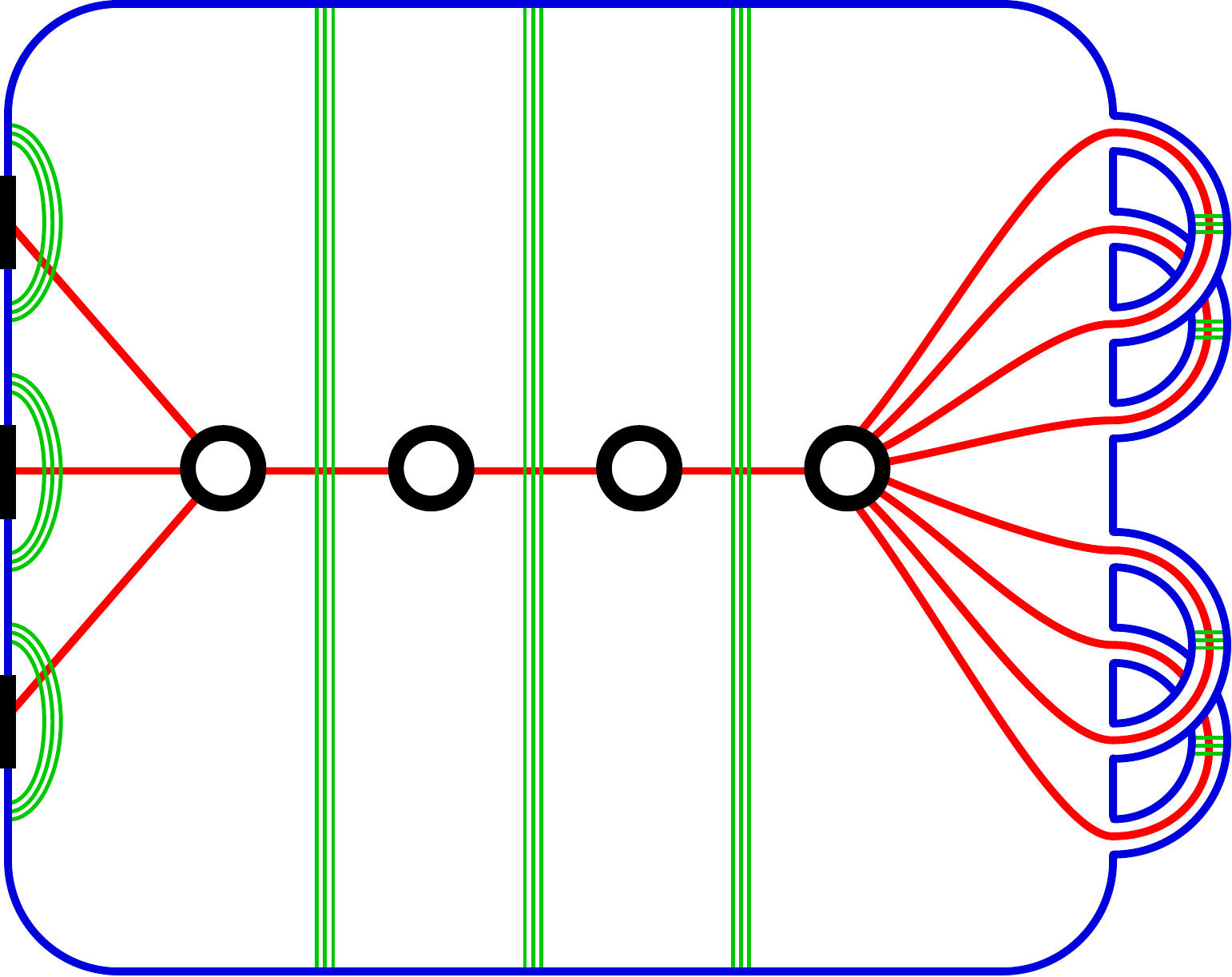}
\caption{The basis for $\bar{H}^{\lf}_m(\mouter)$ is in \textcolor{red}{red} and the basis for $H^{\partial}_m(\mouter)$ is in \textcolor{dgreen}{green}.}
\label{fig-basis-out}
\end{figure}
\begin{defn}(Subspaces in these homology groups)\\
Using these particular homology classes, we define the following subspaces generated by all these elements:
\begin{equation}
\begin{aligned}
\bar{H}^{\lf}_m(\mouter) &\coloneqq \langle \cV_e \mid e \in E_{n-1+k+2g,m} \rangle \subseteq H^{\lf}_m(\mouter)\\
H_m^{\partial}(\mouter) &\coloneqq \langle \cG_f \mid f \in E_{n-1+k+2g,m} \rangle \subseteq H_m(\mouter,\partial \mouter).
\end{aligned}
\end{equation}
\end{defn}
\begin{lem}[Diagonal intersection form]\label{L:2}\hspace{0pt}\\
The intersection pairing is diagonal on these subspaces, in the sense that we have:
\begin{equation}
\begin{aligned}
\pairing &\colon \bar{H}^{\lf}_m(\mouter) \otimes H_m^{\partial}(\mouter) \longrightarrow R\\
& \qquad\quad \langle \cV_e, \cG_f \rangle= \delta_{e,f},
\end{aligned}
\end{equation}
where $\delta$ denotes the Kronecker symbol. As a consequence, the sets of elements $\{\cV_e\}$ and $\{\cG_f\}$ are linearly independent, so they are bases for the free $R$-modules $H_m^{\partial}(\mouter)$ and $\bar{H}^{\lf}_m(\mouter)$.
\end{lem}

This property follows by a similar argument as the one presented in Lemma \ref{L:1}.

Finally, we note that there is no need to restrict to a submodule on the locally-finite side:

\begin{lem}\label{lem:generating}
We have
\[
\bar{H}_m^{\lf}(\minner) = H_m^{\lf}(\minner) \qquad\text{and}\qquad \bar{H}_m^{\lf}(\mouter) = H_m^{\lf}(\mouter).
\]
In other words, the elements $\{\cD_f\}$ generate the locally-finite homology group $H_m^{\lf}(\minner)$ and the the elements $\{\cV_e\}$ generate the locally-finite homology group $H_m^{\lf}(\mouter)$.
\end{lem}
\begin{proof}
This is precisely Theorem \ref{thm-lf-basis}.
\end{proof}

Putting together Theorem \ref{T:1}, Lemma \ref{L:1}, Lemma \ref{L:2} and Lemma \ref{lem:generating}, we conclude the proof of Theorem \ref{thm-pairings}, presented in the introduction.

\section{Embeddings of representations}\label{embeddings}

This part is devoted to the study of the relationships between the two pairs of homology groups presented in the previous section:
$$\left( {\color{red}  H_m^\partial(\minner)  \hspace{5mm} {\color {black} \text { and } } \hspace{5mm} H_m^{\lf}(\mouter)} \right) \ \ \ \ \ \left( {\color{dgreen} H_m^\partial(\mouter)  \hspace{5mm} {\color {black} \text { and } } \hspace{5mm} H_m^{\lf}(\minner) }\right).$$

\begin{thm}[First part of Theorem \ref{thm-diagonal}]\hspace{0pt}\\
We have the follwing $R$-module homomorphisms which are $\Gamma(\Sigma;\bouter)$-equivariant:
\begin{equation}\label{injections2}
\begin{aligned}
\iota_{in} \colon H_m^\partial(\minner) &\too H_m^{\lf}(\mouter) \\
\iota_{out} \colon H_m^\partial(\mouter) &\too H_m^{\lf}(\minner).
\end{aligned}
\end{equation}
\end{thm}
\begin{proof}
We define the map $\iota_{in}$, which comes from the following composition of maps between pairs of spaces:
\begin{center}
\begin{tikzpicture}
[x=1.2mm,y=1.4mm]

\node (b1) at (-30,30)    {$(\minner, \partial \minner)$};
\node (b2) at (-30,15)   {$(M, \partial \minner)$};
\node (b3) at (30,15)   {$(M/ \partial \minner, \star)$};
\node (b4) at (30,30)   {$(\mouter^{\star},\infty)$};

\draw[->] (b1) to node[left,font=\normalsize]{$i$} (b2);
\draw[->] (b2) to node[above,,font=\normalsize]{$f_q$} (b3);
\draw[->] (b3) to node[right,,font=\normalsize]{$f_c$} (b4);
\draw[->, dashed] (b1) to node[above,font=\normalsize]{$\iota_{in}$}  (b4);
\end{tikzpicture}
\end{center}
Here, $i$ is the inclusion from $\minner$ into $M$ and $f_q$ is the map induced by taking the quotient of the space $M$ with respect to the subspace $\partial \minner$, where we write $\star$ for the point $\partial\minner / \partial\minner$.

Note that $M/\partial\minner \smallsetminus \{\star\}$ is equal to $\mouter$, so the composition $\mouter \hookrightarrow M \to M/\partial\minner$ is a one-point \emph{partial} compactification of $\mouter$. This means that the (unique) one-point compactification $\mouter \to \mouter^{\star}$ factors through a map $M/\partial\minner \to \mouter^{\star}$, which is the map $f_c$ above.

Now, using the property that these maps preserve the subspaces $\partial \minner$ and $\partial \mouter$ of $M$, we see that $\iota_{in}$ is $\mathrm{Diff}(\Sigma;\bouter)$-equivariant, and hence its induced map on homology is $\Gamma(\Sigma;\bouter)$-equivariant, which concludes the first part of the statement. 

A similar argument leads to the construction of the map $\iota_{out}$, which is equivariant with respect to the diffeomorphism group action.
\end{proof}

In the next part, we are interested in studying the image of these maps evaluated on the special bases that we have defined in section \ref{pairings}. We start with the following notation.
\begin{defn}(New elements in $H_m^{\lf}(\mouter)$ and $H_m^{\lf}(\minner)$)\\
For any partition $e \in E_{n-1+k+2g,m}$ we consider the classes given by the geometric supports $\bU_e$ and $\bG_e$, this time in the homologies $H_m^{\lf}(\mouter)$ and $H_m^{\lf}(\minner)$ and denote then as below:
\begin{equation}
\begin{aligned}
&\tilde{\cU}_e:= [\bU_e] \in H_m^{\lf}(\mouter) \\
& \tilde{\cG}_e:= [\bG_e] \in  H_m^{\lf}(\minner).
\end{aligned}
\end{equation}

\end{defn}
\begin{rmk} We notice that these new families of elements are exactly the images of the special classes $\cU_e$ and $\cG_e$, seen in the Borel-Moore homology: 
\begin{equation}
\begin{aligned}
& \tilde{\cU}_e= \iota_{in}(\cU_e) \in H_m^{\lf}(\mouter) \\
& \tilde{\cG}_e= \iota_{out}(\cG_e) \in  H_m^{\lf}(\minner).
\end{aligned}
\end{equation}
\end{rmk}
\begin{lem}[Relations between the special homology classes]\hspace{0pt}\\
For any partition $e=(e_1,...,e_{n-1+k+2g})$ from $E_{n-1+k+2g,m}$ the following relations are satisfied:
\begin{equation}
\begin{aligned}
& \tilde{\cU}_e = \prod_{i=1}^{n-1+k+2g} [e_i]_{u}! \, \cV_e \in H_m^{\lf}(\mouter) \\
& \tilde{\cG}_e = \prod_{i=1}^{n-1+k+2g} [e_i]_{u}! \, \cD_e \in  H_m^{\lf}(\minner).
\end{aligned}
\end{equation}
\end{lem}
\begin{proof}
Let us fix $e \in E_{n-1+k+2g,m}$.  In order to prove the first relation presented above, we will use the intersection form from Lemma \ref{L:2} given by:
\begin{equation}
\begin{aligned}
\pairing &\colon H^{\lf}_m(\mouter) \otimes H_m^{\partial}(\mouter) \longrightarrow R\\
& \qquad\quad \langle \cV_e, \cG_f \rangle= \delta_{e,f}.
\end{aligned}
\end{equation}
Since this form has the identity matrix on the bases $\{ \cV_s \}_{s\in E_{n-1+k+2g,m}}$ and $\{ \cG_f \}_{f\in E_{n-1+k+2g,m}}$, we will investigate the intersections:
$$  \langle \tilde{\cU}_e, \cG_f \rangle, \text { for all } f\in E_{n-1+k+2g,m}.$$
By a similar geometric argument as the one presented in the proof of Lemma \ref{L:2}, we notice that if $ \langle \tilde{\cU}_e, \cG_f \rangle \neq 0$, it means that we have at least one intersection point between $\bU_e$ and $\bG_f$ in the configuration space and this forces that the partitions coincide, meaning that $e=f$. 

Now, it remains to compute the intersection $\langle \tilde{\cU}_e, \cG_e \rangle $. We recall (see Remark \ref{rmk:choice}) that the intersection form $\pairing$ depends on a choice of a collection of paths from $\bouter$ to each component of $\binner$. The value of the intersection form is a sum of terms that are indexed by the intersection points of $\bU_e$ and $\bG_e$, with coefficients that are evaluations of the local system on certain loops, based at the boundary and passing through these intersection points, which are constructed using these chosen paths between $\bouter$ and $\binner$.
Having this in mind, it is enough to compute such intersections locally, as in the picture:

\begin{figure}[H]
\centering
\includegraphics[scale=0.3]{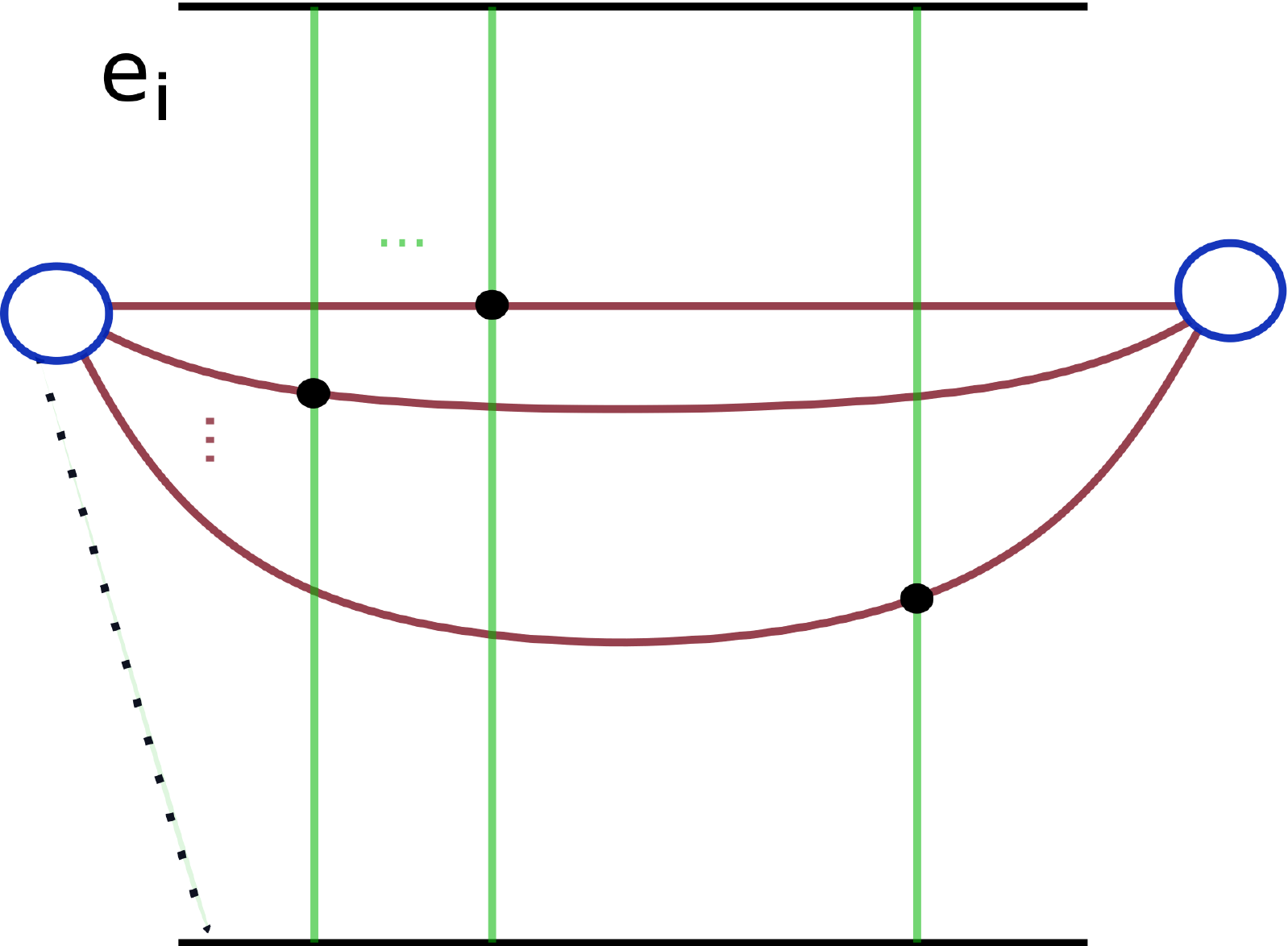}
\caption{Local intersection}
\label{Picture}
\end{figure}

We see that geometrically we have $(e_i)!$ intersection points, and each time, the corresponding path will be evaluated by the local system to a power of the variable $u$. (Recall that we are assuming that the local system is $u$-homogeneous.) Inductively, we can see that the scalar that we get is exactly $$[e_i]_u!.$$
Doing this locally for all $i \in \{1,..,n-1+k+2g\}$, we conclude that:
$$\langle \tilde{\cU}_e, \cG_e \rangle =\prod_{i=1}^{n-1+k+2g}[e_i]_u! $$ and so the first relation from the statement holds.

The second relation follows in an analogous manner, using the pairing from Lemma \ref{L:1}.
\end{proof}

\section{Genericity}\label{genericity}

In this section we sketch the proof of Proposition \ref{thm-genericity}. As remarked in \S\ref{results}, this is essentially due to T.\ Kohno, and our proof is just a mild extension of his, allowing more general (orientable) surfaces and local systems over any unital ring $R$ (rather than $\bC$). We will therefore just sketch how to adapt his proof from \cite[Theorem 3.1]{Kohno2017}.

Recall that we are assuming that the local system $\cL$ on $\minner$ is \emph{generic} (for $\minner$), which means that whenever $\gamma$ is an unbased loop in $\minner$ that may be homotoped to be disjoint from any compact subset, and we write $m_\gamma \in R^\times$ for the monodromy of $\cL$ around $\gamma$, we have
\begin{equation}
\label{eq:genericity}
1 - m_\gamma \in R^\times ,
\end{equation}
in other words, $1 - m_\gamma$ is a \emph{unit} of $R$.

\begin{proof}[Sketch of Proposition \ref{thm-genericity}]
Let $\mathring{M} = C_m(\mathrm{int}(\Sigma))$, which may be seen as a complex affine variety. We may write $\mathring{M} = X \smallsetminus D$, where $X$ is a complex projective variety and $D = D_1 \cup \ldots \cup D_r$ is a union of finitely many normal crossing divisors. Moreover, we may arrange that passing from the space $\mathring{M} = C_m(\mathrm{int}(\Sigma))$ to the larger space $\minner = C_m(\mathrm{int}(\Sigma) \cup \mathrm{int}(\binner))$ corresponds to taking the union with a subset of these divisors, so we still have $\minner = X \smallsetminus D'$, where $D' = D_1 \cup \ldots \cup D_s$ for some $s \leq r$. Let $\gamma_i \subset \minner$ be the normal loop in $X$ of the divisor $D_i$ for $i \in \{1,\ldots,s\}$.

The proof of \cite[Theorem 3.1]{Kohno2017} then goes through to prove that the natural homomorphism $H_m(\minner ; \cL) \to H_m^{\lf}(\minner ; \cL)$ is an isomorphism, as long as, in our setting, the local system $\cL$ satisfies the key property (5) on page 127 of \cite{Kohno2017}. Just as in \cite{Kohno2017}, this follows as long as the monodromy $m_{\gamma_i} \in R^\times$ of the local system around each divisor $D_i$ satisfies the following: if we take $\cL_i$ to be the local system (defined over $R$) on $\bC^* = \bC \smallsetminus \{0\}$ whose monodromy around a generator of $\pi_1(\bC^*)$ is $m_{\gamma_i}$, then we have
\[
H^q(\bC^* ; \cL_i) = 0
\]
for all $q>0$. We therefore just have to verify this property.

We may deformation retract $\bC^*$ onto a cell complex with one $0$-cell and one $1$-cell, so its cellular cochain complex, with respect to the local system $\cL_i$, is
\[
R \too R,
\]
where the first $R$ is in degree $0$, the second $R$ is in degree $1$ and the differential is multiplication by $1 - m_{\gamma_i}$. We have to show that its cohomology in degree $1$ is trivial, in other words, that
\[
R / \langle 1 - m_{\gamma_i} \rangle \;=\; 0.
\]
But the normal loop $\gamma_i$ may be homotoped to be arbitrarily close to the divisor $D_i$, and hence to be disjoint from any given compact subspace of $\minner$. Thus, by our assumption, $1 - m_{\gamma_i}$ is a unit of the ring $R$, so the ideal $\langle 1 - m_{\gamma_i} \rangle$ is equal to $R$ and we are done.
\end{proof}

\section{Locally-finite homology of infinite coverings}\label{coverings}

In \S\ref{Shapiro} we first recall \emph{Shapiro's lemma for the homology of covering spaces} and we discuss its failure for locally-finite homology in \S\ref{Shapiro-lf}. We then prove Theorem \ref{thm-coverings} in \S\ref{proof-Theorem-E} in two parts: Theorem \ref{thm-lf-basis} describes the locally-finite homology of the base space $\minner$ and Theorem \ref{thm-lf-basis-covering} describes the locally-finite homology of the covering $(\minner)^\varphi$, as well as the natural map between them.

\subsection{Shapiro's lemma for covering spaces.}\label{Shapiro}

Let $X$ be a path-connected, based space admitting a universal cover and let $Y \to X$ be any regular covering of $X$, corresponding to a surjection $\pi_1(X) \to G$. Denote by $\cL$ the rank-1 local system $k[Y] \to X$, defined over $k[G]$, given by taking free $k$-modules fibrewise. There is then a naturally-defined isomorphism
\begin{equation}
\label{eq:Shapiro}
H_i(X;\cL) = H_i(X;k[G]) \too H_i(Y;k)
\end{equation}
for any $i\geq 0$. This may be proven directly from the definitions, as follows. The left-hand side is, by definition, the $i$th homology group of the chain complex
\[
C_*(\widetilde{X};k) \otimes_{k[\pi_1(X)]} k[G],
\]
where $\widetilde{X}$ is the universal covering of $X$ and $C_i(Z;k) = k\{\Delta^i \to Z\}$ is the free $k$-module generated by the set of singular $i$-simplices in $Z$, with the differential defined, as usual, as an alternating sum all ways of forgetting one vertex. Now we note that, for any regular covering $Z \to X$ corresponding to a quotient $\pi_1(X) \to H$, there is a natural decomposition
\[
C_i(Z;k) = k\{\Delta^i \to Z\} \cong k[H]\{\Delta^i \to X\}.
\]
We therefore have isomorphisms (which are compatible with the differentials):
\begin{align*}
C_i(\widetilde{X};k) \otimes_{k[\pi_1(X)]} k[G] &\; = \; k\{\Delta^i \to \widetilde{X}\} \otimes_{k[\pi_1(X)]} k[G] \\
&\;\cong\; k[\pi_1(X)]\{\Delta^i \to X\}  \otimes_{k[\pi_1(X)]} k[G] \\
&\;\cong\; k[G]\{\Delta^i \to X\} \\
&\;\cong\; k\{\Delta^i \to Y\} = C_i(Y;k),
\end{align*}
from which Shapiro's lemma \eqref{eq:Shapiro} follows.

A similar argument also shows that, if $A$ is any subpace of $X$, there is an isomorphism
\begin{equation}
\label{eq:Shapiro-relative}
H_i(X,A;\cL) = H_i(X,A;k[G]) \too H_i(Y,Y|_A;k),
\end{equation}
where $Y|_A = p^{-1}(A)$ if we denote the covering by $p \colon Y \to X$.

\subsection{Locally-finite homology.}\label{Shapiro-lf}
If we replace ordinary homology with locally-finite (i.e.~Borel-Moore) homology, Shapiro's lemma is false in general (unless the covering is finitely-sheeted). However, there is still a well-defined \emph{homomorphism}
\begin{equation}
\label{eq:Shapiro-BM}
H_i^{\lf}(X;\cL) = H_i^{\lf}(X;k[G]) \too H_i^{\lf}(Y;k),
\end{equation}
which we now define.

\begin{defn}
Let $C_i^{\lf}(X;k)$ be the $k$-module given by formal sums $\sum_r \lambda_r f_r$ where $\lambda_r \in k$ and $f_r \colon \Delta^i \to X$ is a continuous map, satisfying
\begin{equation}
\label{eq:locally-finite-condition}
\text{each } x \in X \text{ has an open neighbourhood } U \text{ that intersects only finitely many } f_r(\Delta^i).
\end{equation}
There are obvious differentials, given by forgetting vertices, making $C_*^{\lf}(X;k)$ into a chain complex.
\end{defn}

\begin{defn}
For a covering $Z \to X$, define $C_i^{\lfb}(Z;k)$ to be the $k$-module given by formal sums $\sum_r \lambda_r f_r$ where $\lambda_r \in k$ and $f_r \colon \Delta^i \to Z$ is a continuous map, satisfying
\begin{equation}
\label{eq:locally-finite-base-condition}
\text{each } x \in X \text{ has an open neighbourhood } U \text{ where } Z|_U \text{ intersects only finitely many } f_r(\Delta^i).
\end{equation}
There are obvious differentials, given by forgetting vertices, making $C_*^{\lfb}(Z;k)$ into a chain complex.
\end{defn}

\begin{rmk}\label{rmk:subcomplex}
Note that $C_*^{\lfb}(Z;k)$ is a subcomplex of $C_*^{\lf}(Z;k)$, with equality if and only if $Z \to X$ is a finitely-sheeted covering.
\end{rmk}

\begin{rmk}\label{rmk:twisted-lf-homology}
Recall that the untwisted homology $H_i^{\lf}(X;k)$ is -- by definition -- the $i$th homology group of $C_*^{\lf}(X;k)$. Analogously, one may check from the definitions that the twisted homology $H_i^{\lf}(X;k[G])$ is the $i$th homology group of $C_*^{\lfb}(\widetilde{X};k) \otimes_{k[\pi_1(X)]} k[G]$, where $\widetilde{X}$ is the universal cover of $X$.
\end{rmk}

\begin{defn}
\label{def-base-to-covering}
There is a chain map
\[
C_*^{\lfb}(\widetilde{X};k) \otimes_{k[\pi_1(X)]} k[G] \too C_*^{\lfb}(Y;k)
\]
defined, for $f \colon \Delta^i \to \widetilde{X}$ and $g \in G$, by
\[
f \otimes g \longmapsto (\pi \circ f) . g,
\]
where $\pi \colon \widetilde{X} \to Y$ is the projection of the universal cover of $X$ onto the intermediate cover $Y \to X$ and $.$ denotes the natural action of the deck transformation group $G$ of $Y \to X$ on the set of singular $i$-simplices in $Y$. In fact this chain map is an isomorphism, although we will not need this. Composing with the inclusion of $C_*^{\lfb}(Y;k)$ into $C_*^{\lf}(Y;k)$ (see Remark \ref{rmk:subcomplex}) and taking homology, we obtain a homomorphism
\[
H_i^{\lf}(X;k[G]) \too H_i^{\lf}(Y;k),
\]
by Remark \ref{rmk:twisted-lf-homology}. This is the definition of the homomorphism \eqref{eq:Shapiro-BM}.
\end{defn}

\subsection{The proof of Theorem \ref{thm-coverings}.}\label{proof-Theorem-E}

Recall that we have fixed a surjective homomorphism
\[
\varphi \colon \pi_1(M) \too G,
\]
where $M = C_m(\Sigma)$, and we denote the corresponding regular covering by $M^\varphi \to M$, whose deck transformation group is naturally isomorphic to $G$. Denote its restrictions to $\minner$ and $\mouter$ by $(\minner)^\varphi$ and $(\mouter)^\varphi$ respectively. We also choose a unital ring $k$ and set $R=k[G]$. Taking free $k$-modules fibrewise in the covering, we obtain a bundle of $R$-modules
\[
k[M^\varphi] \too M,
\]
which is a rank-1 local system on $M$ defined over $R$. Denote this local system by $\cL$. By abuse of notation, we also denote its restrictions to $\minner \subseteq M$ and $\mouter \subseteq M$ by $\cL$. Taking $X=\minner$ and $X=\mouter$ in Definition \ref{def-base-to-covering}, we have homomorphisms
\begin{equation}\label{eq:base-to-covering}
\begin{aligned}
H_m^{\lf}(\minner ; \cL) &\too H_m^{\lf}((\minner)^\varphi ; k) \\
H_m^{\lf}(\mouter ; \cL) &\too H_m^{\lf}((\mouter)^\varphi ; k).
\end{aligned}
\end{equation}
Theorem \ref{thm-coverings} consists of an explicit description of these homomorphisms, including their domains and co-domains.

We will prove this theorem under a mild assumption on the decomposition of the boundary of $\Sigma$ into $\binner$ and $\bouter$, namely Assumption \ref{assumption-boundary}. It is easy to remove this assumption; the description of the explicit bases is just a little more fiddly to describe.

\begin{thm}[First part of Theorem \ref{thm-coverings}]
\label{thm-lf-basis}
The locally-finite homology groups $H_m^{\lf}(\minner ; \cL)$ and $H_m^{\lf}(\mouter ; \cL)$ are free $R$-modules generated by the sets of elements $\cB_{\mathrm{in}}$ and $\cB_{\mathrm{out}}$ described in Definitions \ref{def:bin} and \ref{def:bout}.
\end{thm}

\begin{rmk}
Theorem \ref{thm-lf-basis} in fact holds for any rank-1 local system on $M$, not only those arising from a regular covering of $M$. We use it in this generality for the proof of Theorem \ref{thm-pairings} in \S\ref{pairings} above (see Lemma \ref{lem:generating}).
\end{rmk}

\begin{thm}[Second part of Theorem \ref{thm-coverings}]
\label{thm-lf-basis-covering}
There are isomorphisms of $k[G]$-modules
\[
H_m^{\lf}((\minner)^\varphi ; k) \;\cong\; \bigoplus_{\cB_{\mathrm{in}}} \prod_{g \in G} k
\qquad\text{and}\qquad
H_m^{\lf}((\mouter)^\varphi ; k) \;\cong\; \bigoplus_{\cB_{\mathrm{out}}} \prod_{g \in G} k.
\]
Under these identifications, the homomorphisms \eqref{eq:base-to-covering} are given by the natural inclusions
\[
H_m^{\lf}(\minner ; \cL) \;\cong\; \bigoplus_{\cB_{\mathrm{in}}} k[G] \; = \; \bigoplus_{\cB_{\mathrm{in}}} \bigoplus_{g \in G} k \;\longhookrightarrow\; \bigoplus_{\cB_{\mathrm{in}}} \prod_{g \in G} k,
\]
and similarly for $\mouter$. In particular they are injective, and bijective if and only if $G$ is finite.
\end{thm}

\begin{rmk}
The idea of the proof of Theorem \ref{thm-lf-basis} is essentially due to Bigelow~\cite[Lemma 3.1]{Bigelow2004HomologicalrepresentationsIwahori}, who proved it in the setting of the $n$-punctured disc (Example \ref{eg:1}). It has also been extended to more general surfaces by An-Ko~\cite[Lemma 3.3]{AnKo2010}.\footnote{We note that the statement of \cite[Lemma 3.3]{AnKo2010} is not quite correct as stated: instead of the locally-finite (Borel-Moore) homology of the covering space (as stated there), their proof applies to the locally-finite homology of the base space twisted by the corresponding local system. Indeed, as we show in Theorem \ref{thm-lf-basis-covering}, the locally-finite homology of the covering space is not a free $k[G]$-module, unless $G$ is finite, but rather a finite direct sum of copies of a certain completion of $k[G]$.} Theorem \ref{thm-lf-basis} above is a further extension, allowing general local systems (not necessarily arising from a covering) and any partition of the boundary of the surface. Moreover, we give a full, detailed proof, which may be readily adapted to prove similar results for more general configuration spaces (of points, or other submanifolds) in an ambient manifold. The proof of Theorem \ref{thm-lf-basis-covering} follows from the observation that one may apply similar ideas directly to the covering spaces $(\minner)^\varphi$ and $(\mouter)^\varphi$, as long as one is careful to take account of the non-compactness in the fibre direction.
\end{rmk}

\begin{proof}[Proof of Theorem \ref{thm-lf-basis}]
Fix a metric $d$ on the surface $\Sigma$. For $\epsilon > 0$, let $M_\epsilon$ be the subspace of $M$ consisting of configurations $\{p_1,\ldots,p_m\}$ such that either $d(p_i,p_j) < \epsilon$ for some $i\neq j$ or $d(p_i,b) < \epsilon$ for some $i$ and some $b \in \bouter$. Note that
\[
\{ M \smallsetminus M_\epsilon \mid \epsilon > 0 \}
\]
is a family of compact subspaces of $\minner$ that is \emph{cofinal}, meaning that for every compact subspace $K \subset \minner$ there is an $\epsilon > 0$ such that $K \subseteq M \smallsetminus M_\epsilon$. We therefore have
\begin{equation}
\label{eq:limit-1}
H_*^{\lf}(\minner) \;\cong\; \underset{\epsilon\to 0}{\mathrm{lim}} \, H_*(M,M_\epsilon).
\end{equation}
Now consider Figure \ref{fig-basis-in} and denote by $M_{\mathrm{green}}$ the subspace of $\minner$ consisting of configurations $\{p_1,\ldots,p_m\}$ such that each $p_i$ lies on one of the open green intervals. By a similar observation as above, we have
\begin{equation}
\label{eq:limit-2}
H_*^{\lf}(M_{\mathrm{green}}) \;\cong\; \underset{\epsilon\to 0}{\mathrm{lim}} \, H_*(M_{\mathrm{green}} , M_{\mathrm{green}} \cap M_\epsilon).
\end{equation}
Moreover, the space $M_{\mathrm{green}}$ is homeomorphic to the disjoint union of a collection of open $m$-discs, with the disjoint union taken over the set of all possible partitions of the number $m$ indexed by the set of green arcs. Also, the local system $\cL$ restricted to $M_{\mathrm{green}}$ must necessarily be trivial, since its path-components are contractible. Thus, the locally-finite homology of $M_{\mathrm{green}}$ in degree $m$ is the direct sum of one copy of the ground ring $R$ for each partition of $m$ indexed by the set of green arcs in Figure \ref{fig-basis-in}. In other words, it is the free $R$-module generated by this set. This is exactly what we would like to show that $H_m^{\lf}(\minner)$ is isomorphic to, so it will suffice to show that the homomorphism
\[
H_*^{\lf}(M_{\mathrm{green}}) \too H_*^{\lf}(\minner)
\]
induced by the proper embedding $M_{\mathrm{green}} \hookrightarrow \minner$ is an isomorphism. By \eqref{eq:limit-1} and \eqref{eq:limit-2} this will follow as long as we show that, for each sufficiently small $\epsilon > 0$, the inclusion
\begin{equation}
\label{eq:inclusion-1}
(M_{\mathrm{green}} , M_{\mathrm{green}} \cap M_\epsilon) \too (M , M_\epsilon)
\end{equation}
induces isomorphisms on relative homology. The rest of the proof will consist in establishing this fact.

Let us consider again Figure \ref{fig-basis-in} and denote by $\Gamma$ the union of $\bouter$ (the black part of $\partial \Sigma$) and the embedded green arcs in $\Sigma$. (The blue part of $\partial\Sigma$ in Figure \ref{fig-basis-in} is $\binner \subseteq \partial \Sigma$, and the red arcs should be ignored.) Let $M_\Gamma$ denote the subspace of $M$ consisting of configurations $\{p_1,\ldots,p_m\}$ such that each $p_i$ lies in $\Gamma \subseteq \Sigma$. Our first observation is that, for any $\epsilon > 0$, the pair
\[
\{ M_{\mathrm{green}} \, , \, M_\Gamma \cap M_\epsilon \}
\]
is an excisive covering of $M_\Gamma$, meaning that they are both open subspaces and their union is all of $M_\Gamma$. By excision, this means that the inclusion
\begin{equation}
\label{eq:inclusion-2}
(M_{\mathrm{green}} , M_{\mathrm{green}} \cap M_\epsilon) \too (M_\Gamma , M_\Gamma \cap M_\epsilon)
\end{equation}
induces isomorphisms on relative homology.

Next, we choose a strong deformation retraction of $\Sigma$ onto $\Gamma$ with some extra properties. Precisely, let
\[
h_t \colon \Sigma \too \Sigma \qquad\text{for } t \in [0,1]
\]
be a homotopy such that
\begin{itemizeb}
\item $h_1 = \mathrm{id}$
\item $h_t$ restricts to the identity on $\Gamma$
\item $\mathrm{im}(h_0) = \Gamma$
\item $h_t$ is non-expanding (in the sense that $d(x,y) \geq d(h_t(x),h_t(y))$ for all $x,y \in \Sigma$)
\item $h_t$ is an embedding for all $t>0$.
\end{itemizeb}
(The first three properties say that $\{h_t\}$ is a strong deformation retraction; the last two are the additional geometric properties that we will need.)

Denote by $\cN_t$ the subspace of $M$ consisting of configurations $\{p_1,\ldots,p_m\}$ where each $p_i$ lies in $\mathrm{im}(h_t)$. In particular, we have $\cN_0 = M_\Gamma$. We want to show that \eqref{eq:inclusion-1} induces isomorphisms on relative homology, and we know that \eqref{eq:inclusion-2} induces isomorphisms on relative homology, so by 2-out-of-3 it suffices to show that
\begin{equation}
\label{eq:inclusion-3}
(\cN_0 , \cN_0 \cap M_\epsilon) \too (M , M_\epsilon)
\end{equation}
induces isomorphisms on relative homology. Now, as an immediate consequence of the properties of the deformation retraction $\{h_t\}$, it induces, for any $t>0$, a deformation retraction of pairs of spaces of $(M,M_\epsilon)$ onto $(\cN_t , \cN_t \cap M_\epsilon)$. Hence the inclusion 
\begin{equation}
\label{eq:inclusion-4}
(\cN_t , \cN_t \cap M_\epsilon) \too (M , M_\epsilon)
\end{equation}
induces isomorphisms on relative homology for all $t>0$. However, we cannot immediately extend this all the way to $t=0$ since $h_0$ is not an embedding (and it cannot be, since $\Sigma$ is not homeomorphic to $\Gamma$), so it cannot induce a deformation retraction of configuration spaces, since two configuration points may collide at time $t=0$. To get around this problem, we use the following device to make it possible to continue the deformation retraction to time $t=0$.

First, by compactness of $\Sigma$, we may choose $\delta > 0$ such that, for every point $p \in \Sigma$, we have
\begin{equation}
\label{eq:diameter-condition}
d(h_\delta(p) , h_0(p)) < \epsilon / 2.
\end{equation}
In other words, no point $p$ of the surface will travel a distance greater than $\epsilon / 2$ during the last $\delta$ seconds of the deformation retraction of $\Sigma$ onto $\Gamma$. Now define $\cN'_\delta$ to be the subspace of $\cN_\delta$ of configurations $\{p_1,\ldots,p_m\}$ satisfying the additional condition that, for all $i\neq j$, we have
\begin{equation}
\label{eq:key-condition}
h_0(h_{\delta}^{-1}(p_i)) \neq h_0(h_{\delta}^{-1}(p_j)).
\end{equation}

We now show that each of the three inclusions
\[
(\cN_0 , \cN_0 \cap M_\epsilon) \too (\cN'_\delta , \cN'_\delta \cap M_\epsilon) \too (\cN_\delta , \cN_\delta \cap M_\epsilon) \too (M , M_\epsilon)
\]
induces isomorphisms on relative homology, which fill finish the proof. The third inclusion admits a deformation retraction induced by the deformation retraction $\{h_t \mid t \in [\delta,1]\}$ (this is a special case of \eqref{eq:inclusion-4}). The first inclusion admits a deformation retraction induced by the deformation retraction $\{ h_t \circ h_{\delta}^{-1} \mid t \in [0,\delta] \}$, where we use the key property \eqref{eq:key-condition} of configurations in $\cN'_{\delta}$ to ensure that this remains well-defined at $t=0$. Finally, to show that the second inclusion induces isomorphisms on relative homology we will show that the pair
\[
\{ \cN'_\delta \, , \, \cN_\delta \cap M_\epsilon \}
\]
is an excisive covering of $\cN_\delta$. It is clear that these are both open subspaces of $\cN_\delta$, so we just have to show that their union covers $\cN_\delta$, which is equivalent to the statement that
\[
\cN_\delta \smallsetminus \cN'_\delta \;\subseteq\; M_\epsilon .
\]
So let $\{p_1,\ldots,p_m\}$ be a configuration in $\cN_\delta \smallsetminus \cN'_\delta$. This means that each $p_i$ lies in $\mathrm{im}(h_\delta)$ and there exist $i \neq j$ such that $h_0(\bar{p}_i) = h_0(\bar{p}_j)$, where we write $\bar{p}_i = h_{\delta}^{-1}(p_i)$ and $\bar{p}_j = h_{\delta}^{-1}(p_j)$. By property \eqref{eq:diameter-condition} and the triangle inequality, we have
\begin{align*}
d(p_i,p_j) &\leq d(p_i,h_0(\bar{p}_i)) + d(h_0(\bar{p}_j),p_j) \\
&= d(h_\delta(\bar{p}_i),h_0(\bar{p}_i)) + d(h_0(\bar{p}_j),h_\delta(\bar{p}_j)) \\
&< \epsilon / 2 + \epsilon / 2 = \epsilon ,
\end{align*}
which implies that $\{p_1,\ldots,p_m\}$ lies in $M_\epsilon$. This completes the proof.
\end{proof}

\begin{rmk}
The deformation retraction $\{h_t\}$ of $\Sigma$ onto $\Gamma$ (the union of the black part of $\partial \Sigma$ and the green embedded arcs in Figure \ref{fig-basis-in}) may be visualised as expanding each of the blue inner boundary-components (and pushing inwards the blue intervals on the outer boundary-component) within their respective regions (bounded by the green and black arcs) of the surface.

The proof of Theorem \ref{thm-lf-basis} in the case of $\mouter$ is exactly analogous to the proof detailed above in the case of $\minner$. In this case $\Gamma$ is the union of the black part of $\partial \Sigma$ and the red embedded arcs in Figure \ref{fig-basis-out}, and the deformation retraction of $\Sigma$ onto this $\Gamma$ may be visualised as compressing the blue part of the outer boundary of the surface onto the union of the red and black arcs and curves.
\end{rmk}

\begin{proof}[Proof of Theorem \ref{thm-lf-basis-covering}]
We will use the notation of the proof of Theorem \ref{thm-lf-basis}. First, recall that $M_{\mathrm{green}} \subseteq \minner$ is the subspace of configurations $\{p_1,\ldots,p_m\}$ where each $p_i$ lies on one of the green open embedded arcs in $\Sigma$ pictured in Figure \ref{fig-basis-in}. As explained in the proof of Theorem \ref{thm-lf-basis}, it is homeomorphic to a disjoint union of open $m$-balls, indexed by $\cB_{\mathrm{in}}$, the set of partitions of the number $m$ indexed by the set of green arcs. Write $(M_{\mathrm{green}})^\varphi$ for the restriction of the covering $M^\varphi \to M$ to $M_{\mathrm{green}}$. Since each path-component of $M_{\mathrm{green}}$ is contractible (in particular simply-connected), $(M_{\mathrm{green}})^\varphi$ is a trivial covering with fibre $G$, so it is a disjoint union of open $m$-balls indexed by $\cB_{\mathrm{in}} \times G$. We therefore have:
\begin{align*}
H_m^{\lf}((M_{\mathrm{green}})^\varphi ; k) &\;\cong\; H_m^{\lf} \left[\, \bigsqcup_{\cB_{\mathrm{in}} \times G} \! \mathring{\bD}^m ; k \right] \\
&\;\cong\; \prod_{\cB_{\mathrm{in}} \times G} \! H_m^{\lf}(\mathring{\bD}^m) ; k) \\
&\;\cong\; \prod_{\cB_{\mathrm{in}} \times G} \! k \;=\; \bigoplus_{\cB_{\mathrm{in}}} \prod_{G} k \;=\; \bigoplus_{\cB_{\mathrm{in}}} k[[G]].
\end{align*}
The second isomorphism uses the fact that locally-finite homology converts arbitrary disjoint unions into products (just like cohomology) and on the third line we use the fact that $\cB_{\mathrm{in}}$ is finite to rewrite a product as a direct sum. It will therefore suffice to prove that the homomorphism
\begin{equation}
\label{eq:inclusion-5}
H_m^{\lf}((M_{\mathrm{green}})^\varphi ; k) \too H_m^{\lf}((\minner)^\varphi ; k)
\end{equation}
induced by the proper embedding $(M_{\mathrm{green}})^\varphi \hookrightarrow (\minner)^\varphi$ is an isomorphism.

Similarly to the proof of Theorem \ref{thm-lf-basis}, we will interpret the locally-finite homology $H_m^{\lf}((\minner)^\varphi ; k)$ as the inverse limit of relative homology groups $H_m(M^\varphi , A ; k)$ as $A$ varies over a family of subspaces having the property that $M^\varphi \smallsetminus A$ forms a cofinal family of compact subspaces of $(\minner)^\varphi$. Write $M_{\epsilon}^\varphi$ for the restriction of the covering $M^\varphi \to M$ to $M_\epsilon$. In contrast to the proof of Theorem \ref{thm-lf-basis}, the family of subspaces
\[
\{ M^\varphi \smallsetminus M_{\epsilon}^\varphi \mid \epsilon > 0 \}
\]
is \emph{not} a cofinal family of compact subspaces of $(\minner)^\varphi$, because $M^\varphi \smallsetminus M_{\epsilon}^\varphi$ is not compact (unless the covering is finitely-sheeted). We therefore have to work a little harder to construct the necessary family of subspaces.

Since $M^\varphi \to M$ is a regular covering (a principal $G$-bundle for discrete $G$), there is a properly discontinuous action of $G$ on $M^\varphi$ and a homeomorphism $M \cong (M^\varphi)/G$ compatible with the two projections. Choose a fundamental domain $D \subseteq M^\varphi$ for the action of $G$ and write $\bar{D}$ for its closure. For a finite subset $S \subseteq G$, define
\[
L_S \; = \; M^\varphi \smallsetminus \bigcup_{g \in S} \bar{D}.g ,
\]
and note that this is an open subspace of $M^\varphi$. The family of subspaces
\[
\{ L_S \cap (\minner)^\varphi \mid S \text{ finite subset of } G \}
\]
may be thought of as a ``neighbourhood basis of infinity'' for $(\minner)^\varphi$ in the vertical (fibre) direction, in the same way that the family of subspaces
\[
\{ M_{\epsilon}^\varphi \mid \epsilon > 0 \}
\]
may be thought of as a ``neighbourhood basis of infinity'' for $(\minner)^\varphi$ in the horizontal direction. Taking unions and indexing over all pairs $(\epsilon , S)$ with $\epsilon > 0$ and $S \subseteq G$ a finite subset, we obtain a full ``neighbourhood basis of infinity'' for $(\minner)^\varphi$, in the sense that
\[
\{ M^\varphi \smallsetminus (L_S \cup M_{\epsilon}^\varphi) \mid \epsilon > 0 , \, S \text{ finite subset of } G \}
\]
is a cofinal family of compact subspaces of $(\minner)^\varphi$. Hence we have
\[
H_*^{\lf}((\minner)^\varphi) \;\cong\; \underset{\epsilon , S}{\mathrm{lim}} \, H_*(M^\varphi , L_S \cup M_{\epsilon}^\varphi)
\]
and similarly
\[
H_*^{\lf}((M_{\mathrm{green}})^\varphi) \;\cong\; \underset{\epsilon , S}{\mathrm{lim}} \, H_*((M_{\mathrm{green}})^\varphi , (M_{\mathrm{green}})^\varphi \cap (L_S \cup M_{\epsilon}^\varphi)),
\]
where we have now dropped the coefficients $k$ from the notation. In order to prove that \eqref{eq:inclusion-5} is an isomorphism and finish the proof, we therefore just have to show that the inclusion
\begin{equation}
\label{eq:inclusion-6}
((M_{\mathrm{green}})^\varphi , (M_{\mathrm{green}})^\varphi \cap (L_S \cup M_{\epsilon}^\varphi)) \too (M^\varphi , L_S \cup M_{\epsilon}^\varphi)
\end{equation}
induces isomorphisms on relative homology for every $\epsilon > 0$ and finite subset $S \subseteq G$. We will do this by lifting the arguments from the proof of Theorem \ref{thm-lf-basis}, where we had to show that
\[
(M_{\mathrm{green}} , M_{\mathrm{green}} \cap M_\epsilon) \too (M , M_\epsilon)
\]
induces isomorphisms on relative homology. We did this by factoring the inclusion as
\begin{equation}
\label{eq:inclusion-7}
M_{\mathrm{green}} \too M_\Gamma = \cN_0 \too \cN'_{\delta} \too \cN_\delta \too M,
\end{equation}
and considering the pairs of spaces given by intersecting each subspace with $M_\epsilon$. Similarly, we now factor the inclusion \eqref{eq:inclusion-6} by taking the preimages of all of the subspaces \eqref{eq:inclusion-7} under the covering $M^\varphi \to M$ and considering the pairs of spaces given by intersecting each subspace with $L_S \cup M_{\epsilon}^\varphi$. This gives us four inclusions of pairs of spaces, and we have to show that each of these induces isomorphisms on relative homology.

In the proof of Theorem \ref{thm-lf-basis}, for the first and third inclusions of \eqref{eq:inclusion-7}, this was proven by excision arguments, and for the second and fourth inclusions it was proved by a deformation retraction of pairs. The two excision arguments carry over identically to the corresponding inclusions in the covering space -- in each case, one has to show that a certain pair of subsets forms an excisive covering, which is true for the same reasons as in the base space.

In order to lift the deformation retraction to the covering space, to complete the argument for the other two inclusions, we have to be slightly more careful. The deformation retraction $\{h_t\}$ of the surface $\Sigma$ onto $\Gamma$ from the proof of Theorem \ref{thm-lf-basis} has the property that each $h_t$ is \emph{non-expanding} (it does not increase distances between points), and this was essential to ensure that the induced deformation retraction of the configuration space $M$ takes the subspace $M_\epsilon$ into itself, so that it is a deformation retraction of \emph{pairs} of spaces.

The deformation retraction of the configuration space $M$ lifts to a deformation retraction of the covering space $M^\varphi$, by covering space theory, and it automatically has the property that it takes $M_{\epsilon}^\varphi$ into itself. It remains to show that it additionally has the property that it takes $L_S$ into itself, so that we again have a deformation retraction of \emph{pairs} of spaces, and so all of the arguments of the proof of Theorem \ref{thm-lf-basis} involving the deformation retraction will lift to the covering space and finish the proof.

To ensure this, we make sure that we choose the fundamental domain $D \subseteq M^\varphi$ compatibly with the deformation retraction, so that whenever $p \in M^\varphi \smallsetminus \bar{D}$, the path followed by $p$ under the lift of the deformation retraction remains in $M^\varphi \smallsetminus \bar{D}$, i.e., it remains outside of the closure of the fundamental domain. Since the deformation retraction of the covering space is a lift of a deformation retraction of the base, it is equivariant under the action of $G$, so the same property holds when replacing $\bar{D}$ with $\bar{D}.g$ for any $g \in G$. Since $L_S$ is an intersection of subspaces of the form $M^\varphi \smallsetminus \bar{D}.g$, this implies that the deformation retraction takes $L_S$ into itself, as required.
\end{proof}

\subsection{Some elements of the locally-finite homology of the covering}\label{elements}

\begin{defn}
\label{def:helix}
Choose $m$ pairwise disjoint embeddings $\gamma_1,\ldots,\gamma_m \colon S^1 \hookrightarrow \Sigma \smallsetminus \binner$. These determine an embedding
\[
\gamma \colon (S^1)^m \longhookrightarrow \minner .
\]
Suppose that, for each $i \in \{1,\ldots,m\}$, the loop in $\minner$ given by restricting $\gamma$ to the $i$-th copy of $S^1$ does \emph{not} lift to a loop in the covering $M^\varphi \to M$, and neither does any positive power of this loop. Denote the preimage of $\gamma((S^1)^m)$ under $(\minner)^\varphi \to \minner$ by
\[
\mathrm{helix}(\gamma) \subseteq (\minner)^\varphi .
\]
The name comes from the fact that this is, intuitively, a product of $m$ infinite helices covering the embedded $m$-torus in $\minner$. This is an orientable, properly-embedded submanifold (homeomorphic to $\bR^m$), so its fundamental class is an element
\[
[\mathrm{helix}(\gamma)] \in H_m^{\lf}((\minner)^\varphi ; k).
\]
When it is non-zero, this element does not lie in the image of the homomorphism \eqref{eq:base-to-covering}.
\end{defn}

\begin{eg}
As a simple example, take $m=1$ and consider the two green embedded circles in Figure \ref{fig-spiral}. (They are drawn as spirals to indicate that we will lift them to infinite helices.) If we take $\gamma_1$ in Definition \ref{def:helix} to be the left-hand circle, we have $[\mathrm{helix}(\gamma)] = 0$. This is because the left-hand green circle, and hence its preimage in $(\minner)^\varphi$, may be homotoped arbitrarily close to the black boundary-component that it encircles, which is a non-compact end of $(\minner)^\varphi$, so it is nullhomologous as a locally-finite cycle. On the other hand, if we take $\gamma_1$ in Definition \ref{def:helix} to be the right-hand circle, we have
\begin{equation}
\label{eq:helix-formula}
[\mathrm{helix}(\gamma)] = \cV_e . \sum_{i \in \bZ} (1-y)(yz)^i,
\end{equation}
where $\cV_e$ denotes the fundamental class of the properly embedded open red interval in Figure \ref{fig-spiral} and $y,z \in G$ denote the monodromy of the covering around small loops encircling the two black boundary-components inside the right-hand green circle.
\end{eg}

\begin{figure}[ht]
\centering
\includegraphics[scale=0.4]{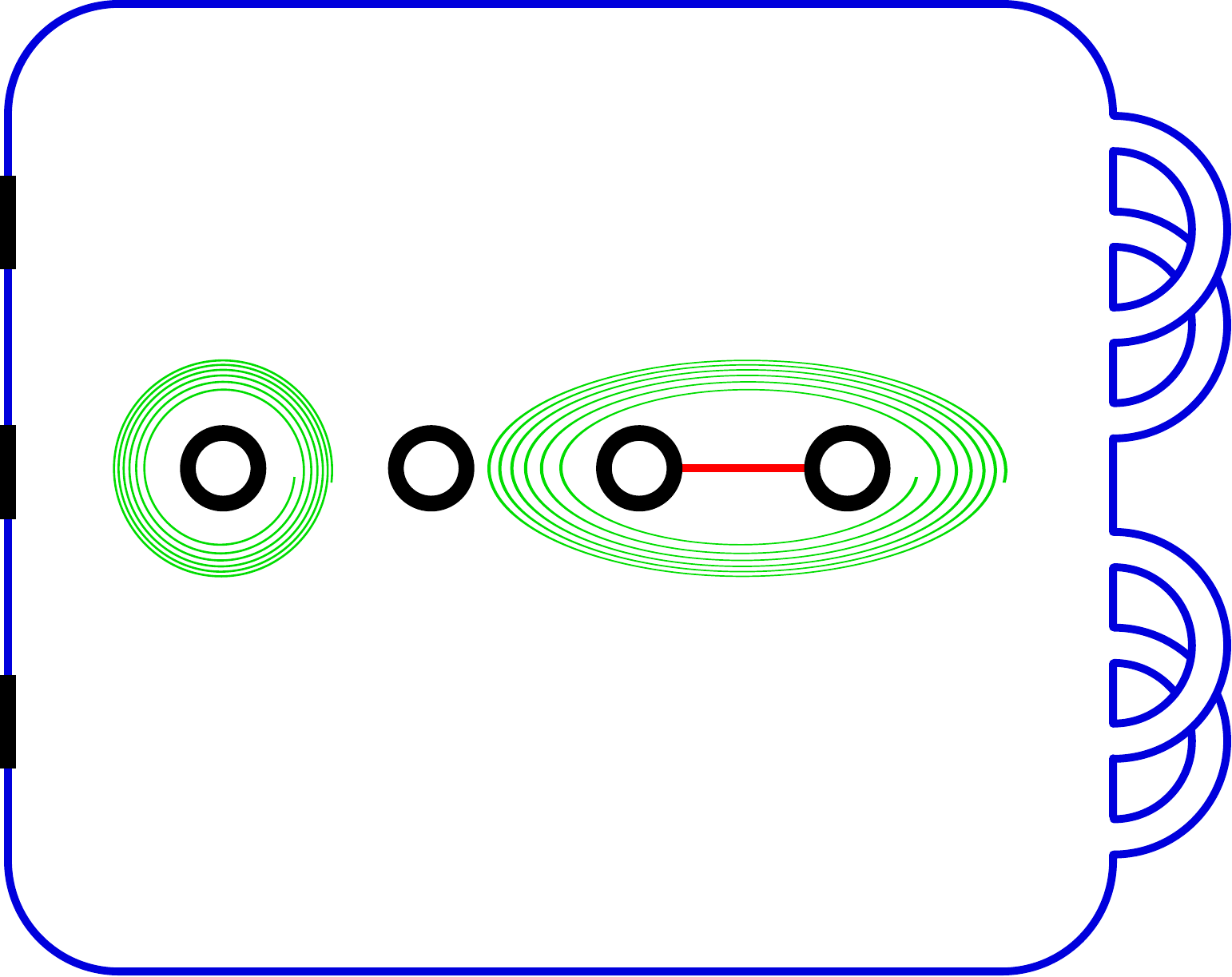}
\caption{Two elements of $H_m^{\lf}((\minner)^\varphi ; k)$ given by infinite helices as in Definition \ref{def:helix}. The left-hand helix is nullhomologous, but the right-hand helix is non-trivial and given by the formula \eqref{eq:helix-formula} in terms of the identification of Theorem \ref{thm-lf-basis-covering}. In particular, it lies outside of the image of the homomorphism \eqref{eq:base-to-covering}.}
\label{fig-spiral}
\end{figure}

\section{Formula for the intersection form in the covering}\label{formula}

In this part, we restrict to the case when the local system comes from a covering space. In this situation, we aim to relate the homology twisted by this local system with the untwisted homology of the corresponding covering space. 
First, we define certain special elements in the homology of the covering space. In order to do this, we will fix some collections of paths from the boundary to the submanifolds that we have defined in the previous sections. Let us make this precise. We fix a basepoint $\bf d$ on the boundary of the configuration space and let $\tilde{\bf d}$ be a lift of it in the $\varphi$-covering. Next, we consider a path in the configuration space from the basepoint $\bf d$ 
towards each of the following basic manifolds:
\[
{\bV}_e, \ {\bG}_e, \ {\bU}_e, \ {\bD}_e.
\]
We do this by drawing $n-1+k+2g$ segments on our surface, and then considering the image of their product quotiented to the configuration space. We denote these paths by:
$$\eta^V_e, \ \eta^G_e, \ \eta^U_e, \ \eta^D_e.$$
Now, we define the paths $$\tilde{\eta}^V_e, \ \tilde{\eta}^G_e, \ \tilde{\eta}^U_e, \ \tilde{\eta}^D_e$$ to be the unique lifts of these paths in the corresponding coverings, through the point $\tilde{\bf d}$.

\begin{defn}(Special elements in the homology of the coverings)\\
Let us consider the unique lifts of the submanifolds ${\bV}_e,{\bG}_e,{\bU}_e,{\bD}_e$ in the coverings corresponding to $\varphi$, through the points $$\tilde{\eta}^V_e(1), \tilde{\eta}^G_e(1), \tilde{\eta}^U_e(1), \tilde{\eta}^D_e(1)$$ respectively, and denote their homology classes by:
\begin{equation}
\begin{aligned}
[\tilde{\bV}_e] &\subseteq H^{\lf}_m\left((\mouter)^{\varphi}\right)\\
[\tilde{\bG}_e] &\subseteq H_m \left( (\mouter)^{\varphi},\partial (\mouter)^{\varphi}\right)\\
[\tilde{\bU}_e] &\subseteq H^{\lf}_m\left((\minner)^{\varphi}\right)\\
[\tilde{\bD}_e] &\subseteq H_m \left( (\minner)^{\varphi},\partial (\minner)^{\varphi}\right).
\end{aligned}
\end{equation}
\end{defn}
\begin{defn}(Subspaces in the homology of the covering generated by special classes)\\
We define the subspaces generated by these families in the homology of the coverings $(\minner)^\varphi$ and $(\mouter)^\varphi$ as below:
\begin{equation}\label{eq:VGUD}
\begin{aligned}
H^{\lf,f}_m((\mouter)^\varphi) &\coloneqq \langle [\tilde{\bV}_e] \mid e \in E_{n-1+k+2g,m} \rangle \subseteq H^{\lf}_m\left((\mouter)^{\varphi}\right)\\
H_m^{\partial,f}((\mouter)^\varphi) &\coloneqq \langle [\tilde{\bG}_e] \mid e \in E_{n-1+k+2g,m} \rangle \subseteq H_m \left( (\mouter)^{\varphi},\partial (\mouter)^{\varphi}\right)\\
H^{\lf,f}_m((\minner)^\varphi) &\coloneqq \langle [\tilde{\bU}_e] \mid e \in E_{n-1+k+2g,m} \rangle \subseteq H^{\lf}_m\left((\minner)^{\varphi}\right)\\
H_m^{\partial,f}((\minner)^\varphi) &\coloneqq \langle [\tilde{\bD}_e] \mid e \in E_{n-1+k+2g,m} \rangle \subseteq H_m \left( (\minner)^{\varphi},\partial (\minner)^{\varphi}\right).
\end{aligned}
\end{equation}
\end{defn}

Now we use the results presented in section \S\ref{coverings} concerning the relation between the twisted homology and the homology of the corresponding covering space. More specifically, the choice of the paths in the base configuration space determines the following corresponding morphisms between these two types of homologies, which preserve the special elements accordingly.

\begin{prop}[Equivariant injections between twisted homologies and homologies of coverings]\label{prop:equivariant-injections}
There are isomorphisms as follows, equivariant with respect to the $\Gamma(\Sigma, \bouter)$-actions:
\begin{equation}\label{eq:injection-lf-out}
\begin{aligned}
H_m^{\lf}(\mouter) &\too H_m^{\lf,f}((\mouter)^\varphi)\\
{\cV}_e & \dashrightarrow [\tilde{\bV}_e]
\end{aligned}
\end{equation}
\begin{equation}\label{eq:injection-bd-out}
\begin{aligned}
H_m^{\partial}(\mouter) &\too H_m^{\partial,f}((\mouter)^\varphi)\\
{\cG}_e & \dashrightarrow [\tilde{\bG}_e]
\end{aligned}
\end{equation}
\begin{equation}\label{eq:injection-lf-in}
\begin{aligned}
H_m^{\lf}(\minner) &\too H_m^{\lf,f}((\minner)^\varphi)\\
{\cU}_e & \dashrightarrow [\tilde{\bU}_e]
\end{aligned}
\end{equation}
\begin{equation}\label{eq:injection-bd-in}
\begin{aligned}
H_m^{\partial}(\minner) &\too H_m^{\partial,f}((\minner)^\varphi)\\
{\cD}_e & \dashrightarrow [\tilde{\bD}_e].
\end{aligned}
\end{equation}
\end{prop}

\begin{proof}
Injectivity of the maps \eqref{eq:injection-lf-out} and \eqref{eq:injection-lf-in} follows from Theorem \ref{thm-lf-basis-covering}. Injectivity (and indeed bijectivity) of the maps
\[
H_m(\mouter,\partial\mouter) \too H_m((\mouter)^\varphi , \partial(\mouter)^\varphi)
\]
(and similarly for $\minner$) follows from Shapiro's lemma -- see \S\ref{Shapiro}, and hence their restrictions to \eqref{eq:injection-bd-out} and \eqref{eq:injection-bd-in} are also injective. The maps act as indicated on the special homology classes by construction. Surjectivity then follows from \eqref{eq:VGUD}.
\end{proof}

\begin{coro}\label{CorADO}
There are $\Gamma(\Sigma, \bouter)$-equivariant injective homomorphisms as follows:
\begin{equation}
\begin{aligned}
H_m^{\partial,f}((\minner)^\varphi) &\longhookrightarrow H_m^{\lf,f}((\mouter)^\varphi), \\
H_m^{\partial,f}((\mouter)^\varphi) &\longhookrightarrow H_m^{\lf,f}((\minner)^\varphi).
\end{aligned}
\end{equation}
\end{coro}
\begin{proof}
This is a consequence of Theorem \ref{thm-diagonal} and the isomorphisms of Proposition \ref{prop:equivariant-injections}.
\end{proof}

\begin{rmk}\label{RkADO}
We use Corollary \ref{CorADO} in \cite{Anghel2019, Anghel2020, Anghel2020a}, where we pass from the subspace generated by multiforks in the Borel-Moore homology to the subspace generated by the same underlying manifolds in the Borel-Moore homology relative just to the ``part of infinity'' concerning the punctures. The former is exactly the homology relative to the boundary of little discs around punctures, so exactly $H_m^{\partial,f}((\minner)^\varphi)$. The injectivity of this result, together with the equivariance, ensures for us that in the corresponding context (punctured disk, with genus zero and $k=0$ or $k=1$), the braid group actions correspond in the two homologies. 
\end{rmk}

\subsection{Formula for the intersection form in the covering}
Finally, we show that these subrepresentations in the homology of covering spaces are related by intersection pairings, which can be computed explicitly. Combining the results on intersection pairings from Theorem \ref{thm-pairings} with the isomorphisms of Proposition \ref{prop:equivariant-injections}, we have the following.
\begin{prop}\label{prop-pairings}
There are \textbf{non-degenerate} $\Gamma(\Sigma;\bouter)$-invariant pairings:
\begin{align}
\pairing \colon H_m^{\lf,f}((\minner)^{\varphi}) \otimes H_m^{\partial,f}((\minner)^{\varphi}) &\too R, \label{eq:pairing-in-covering} \\
\pairing \colon H_m^{\lf,f}((\mouter)^{\varphi}) \otimes H_m^{\partial,f}((\mouter)^{\varphi}) &\too R.
\end{align}
\end{prop}

\paragraph{Description of the intersection form}
\hspace{0pt}\\
Let us start with two homology classes $[X]\in H_m^{\lf,f}((\minner)^{\varphi})$ and $[Y] \in H_m^{\partial,f}((\minner)^{\varphi})$. Furthermore, we assume that these classes are represented by immersed submanifolds
\[
\tilde{X}, \tilde{Y} \looparrowright (\minner)^\varphi
\]
that intersect transversely in finitely many points. Write $X$ and $Y$ respectively for their projections onto $\minner$. For each intersection point $x \in X \cap Y$, we will construct an associated loop $l_x \subset \minner$. Let $\alpha_x$ be the sign of the geometric intersection of $X$ and $Y$ at $x$.

a) {\bf Construction of $l_x$}\\
We choose two paths $\gamma_{X}, \gamma_{Y}$ in $\minner$ starting at $\bf d$ and ending on $X$ and $Y$ respectively, such that moreover $\tilde{\gamma}_{X}(1) \in \tilde{X}$ and $\tilde{\gamma}_{Y}(1) \in \tilde{Y}$, where $\tilde{\gamma}_X$ and $\tilde{\gamma}_Y$ are the lifts of $\gamma_X$ and $\gamma_Y$ starting at $\tilde{\bf d}$. We then choose paths $\delta_{X}, \delta_{Y} \colon [0,1]\rightarrow \minner$ such that:
\begin{equation}
\begin{cases}
\mathrm{Im}(\delta_{X})\subseteq X; \delta_{X}(0) = \gamma_{X}(1);  \delta_{X}(1) = x\\
\mathrm{Im}(\delta_{Y})\subseteq Y; \delta_{Y}(0) = \gamma_{Y}(1);  \delta_{Y}(1) = x.
\end{cases}
\end{equation}
Now we define the loop $l_x \subset \minner$ by $l_x=\gamma_{Y}\circ\delta_{Y}\circ (\delta_{X})^{-1}\circ (\gamma_{X})^{-1}$.

b) {\bf Formula for the intersection form}\\
The intersection form \eqref{eq:pairing-in-covering} from Proposition \ref{prop-pairings} may then be computed using these loops and the local system, as follows:
\begin{equation}\label{eq:geometric-formula}
\langle [\tilde{X}],[\tilde{Y}]\rangle \; = \sum_{x \in X \cap Y} \alpha_x \cdot \varphi(l_x) .
\end{equation}


\phantomsection
\addcontentsline{toc}{section}{References}
\renewcommand{\bibfont}{\normalfont\small}
\setlength{\bibitemsep}{0pt}
\printbibliography


\vspace{0pt plus 1filll}

\noindent {\itshape Mathematical Institute, University of Oxford, Woodstock Road, Oxford, OX2 6GG, United Kingdom}

\noindent {\tt palmeranghel@maths.ox.ac.uk}

\noindent \href{http://www.cristinaanghel.ro/}{www.cristinaanghel.ro}

\noindent {\itshape Institutul de Matematică Simion Stoilow al Academiei Române (IMAR), 21 Calea Griviței, 010702 București, Romania}

\noindent {\tt mpanghel@imar.ro}

\noindent \href{https://mdp.ac}{mdp.ac}

\end{document}